\documentclass[11pt,letterpaper]{article}
\usepackage{amsmath}
\usepackage{amssymb}
\usepackage{amsthm}
\usepackage{amscd}
\usepackage{amsfonts}
\usepackage{graphicx}
\usepackage{fancyhdr}
\usepackage{multirow}
\usepackage{array}
\usepackage[usenames,dvipsnames]{color}
\usepackage{tikz}
\usetikzlibrary{shapes,arrows}
\usepackage{wrapfig}
\usepackage{tabularx}
\usepackage{xparse}
\usepackage{xcolor}
\usepackage{framed}
\usepackage{enumitem}
\usepackage{hyperref}

\hypersetup{
  colorlinks=true,
  linkcolor=blue,
  citecolor=red,
  urlcolor=purple
}
\usepackage{biblatex}
\usepackage{comment}
\usepackage{subfigure}
\usepackage{titling}
\allowdisplaybreaks

\newtheorem*{theorem*}{Theorem}
\newtheorem*{conjecture*}{Conjecture}
\newtheorem*{proposition*}{Proposition}

\usepackage[sc]{titlesec}
\titleformat*{\section}{\large\bfseries\sc\filcenter}
\titleformat*{\subsection}{\normalfont\bfseries\sc\filcenter}

\usepackage{cleveref} 

\newcounter{itemC}
\NewDocumentCommand{\printNum}{s}{%
    \stepcounter{itemC}
    \IfBooleanTF{#1}%
    {%
        \theitemC.\makebox[0pt]{*}%
    }{%
        \theitemC.%
    }%
    }
    
 \definecolor{shadecolor}{RGB}{210,210,210}

\tikzstyle{decision} = [diamond, draw, text width=4.5em, text badly centered, node distance=3cm, inner sep=0pt]
\tikzstyle{block} = [rectangle, draw, text width=5em, text centered, rounded corners, minimum height=4em]

\theoremstyle{plain} \numberwithin{equation}{section}
\newtheorem{theorem}{Theorem}[section]
\newtheorem{corollary}[theorem]{Corollary}
\newtheorem{conjecture}{Conjecture}
\newtheorem{problem}{Problem}
\newtheorem{lemma}[theorem]{Lemma}
\newtheorem{proposition}[theorem]{Proposition}

\theoremstyle{definition}
\newtheorem{definition}[theorem]{Definition}

\newtheorem{remark}[theorem]{Remark}

\newcommand{\be}{\begin{enumerate}}
\newcommand{\ee}{\end{enumerate}}
\newcommand{\R}{\mathbb{R}}

\newcommand{\Z}{\mathbb{Z}}

\newcommand*{\fullref}[1]{\hyperref[{#1}]{\Cref*{#1}}}

\renewcommand\footnotemark{}

\begin{document}
\title{\large{\textbf{Bounds in Simple Hexagonal Lattice and Classification of 11-stick Knots}}\thanks{This work was done as part of the \href{https://geometrynyc.wixsite.com/polymathreu}{Polymath Jr. Program} in summer 2022, advised by Dr. Marion Campisi and Nicholas Cazet.}
}

\author{
        \makebox[.25\textwidth]{Yueheng Bao}\\ \small University of Science and \\ \small Technology of China
        \and 
        \makebox[.25\textwidth]{Ari Benveniste}\\ \small Pomona College
        \and 
        \hfill\makebox[.25\textwidth]{Marion Campisi}\\ \small San Jose State University
        \and
        \makebox[.4\textwidth]{Nicholas Cazet}\\ \small University of California, Davis
        \and
        \hfill\makebox[.4\textwidth]{Ansel Goh}\\ \small University of Washington
        \and
        \makebox[.4\textwidth]{Jiantong Liu}\\ \small University of California,\\\small Los Angeles
        \and
        \hfill\makebox[.4\textwidth]{Ethan Sherman}\\ \small Vanderbilt University
}

\usethanksrule
    \maketitle
    
    \begin{abstract}
        The \textit{stick number} and the \textit{edge length} of a knot type in the simple hexagonal lattice (sh-lattice) are the minimal numbers of sticks and edges required, respectively, to construct a knot of the given type in sh-lattice. By introducing a linear transformation between lattices, we prove that for any given knot both values in the sh-lattice are strictly less than the values in the cubic lattice. Finally, we show that the only non-trivial $11$-stick knots in the sh-lattice are the trefoil knot ($3_1$) and the figure-eight knot ($4_1$).
    \end{abstract}

\section{Introduction}
Knots in the cubic lattice have been studied more thoroughly compared to knots in the simple hexagonal lattice because of the simplicity of the former lattice's structure. Some examples of the work done about the knots in the cubic lattice include \cite{adams2012stick}, \cite{diao1994number}, \cite{huang2017lattice}, \cite{huh2010knots}, \cite{huh2005lattice}, and \cite{scharein2009bounds}. Therefore, it is reasonable to examine the properties of the simple hexagonal lattice as well as their relationships with analogous properties in the cubic lattice. For example, \cite{bailey2015stick} has done some work in this direction. 

We are able to show the following results regarding stick numbers and edge lengths, two properties involved in the construction of a knot in a lattice. 

\begin{theorem*}
For any knot type $[K]$, $s_{sh}[K] < s_{L}[K]$, where $s_L$ and $s_{sh}$ are the stick numbers of $[K]$ in the cubic lattice and in the simple hexagonal lattice, respectively. 
\end{theorem*}

\begin{theorem*}
For any knot type $[K]$, $e_{sh}[K] < e_{L}[K]$, where $e_L$ and $s_{sh}$ are the edge lengths of $[K]$ in the cubic lattice and in the simple hexagonal lattice, respectively.
\end{theorem*}

In addition, we give a classification of the sh-lattice knots with stick number $11$. 

\begin{theorem*}
The only non-trivial $11$-stick knots in the sh-lattice are $3_1$ and $4_1$.  
\end{theorem*}

In this paper, we begin by discussing preliminary definitions and notations in \fullref{definitions} before moving into \fullref{primary tool}, where we discuss our primary tool: the linear transformation between the cubic lattice and simple hexagonal lattice. Then, in \fullref{upper bound results} we use this tool to find the upper bounds in the sh-lattice as described by the theorem and proposition above. We also find a lower bound of edge number in the sh-lattice in terms of the cubic lattice. Finally, \fullref{specific knots} is devoted to the classification of $11$-stick knots in sh-lattice, where we show that every non-trivial knot with $11$ sticks has to be either the trefoil knot, $3_1$, or the figure-eight knot, $4_1$. 

\section{Preliminary Definitions}\label{definitions}

A \textit{point lattice} in $\R^3$ is the set of integral linear combinations of vectors $x, y, z$, i.e. $\{ax + by + cz \, | \, a, b, c \in \Z \}$, such that $\{x, y, z\}$ is a basis for $\R^3$. The \textit{cubic lattice} is the point lattice with a basis of $x = \langle 1, 0, 0 \rangle, y = \langle 0, 1, 0 \rangle$, and $z = \langle 0, 0, 1 \rangle$. The \textit{simple hexagonal lattice (sh-lattice)} is the point lattice with a basis of $x = \langle 1, 0, 0 \rangle, y = \langle \frac{1}{2},\frac{\sqrt{3}}{2},0 \rangle,$ and $w = \langle 0, 0, 1 \rangle$. For convenience, we define $z = y - x = \langle - \frac{1}{2}, \frac{\sqrt{3}}{2}, 0\rangle$. We denote the cubic lattice by $\mathbb{L}^3$ and the sh-lattice by $sh$. An \textit{edge} is a line segment between points whose coordinates differ by a basis vector. A \textit{polygon} $\mathcal{P}$ in a point lattice is a continuous path consisting of edges. A maximal line segment parallel to an axis is called a \emph{stick}. We can label a stick by the axis to which it is parallel. For example, a stick parallel to the $x$-axis is an $x$-stick. A \emph{cubic lattice knot} is then a non-intersecting polygon in the cubic lattice consisting of $x$-, $y$-, and $z$-sticks. For the convenience of counting stick numbers in a given cubic lattice polygon $\mathcal P_L$, we define $|\mathcal P_x|$, $|\mathcal P_y|$, and $|\mathcal P_z|$ to be the number of $x$-, $y$-, and $z$-sticks used in the polygon $\mathcal P_L$. The total number of sticks $|\mathcal P_L|$ in a cubic lattice polygon $\mathcal P_L$ is therefore defined by $|\mathcal P_L| = |\mathcal P_x|+|\mathcal P_y|+|\mathcal P_z|$. According to \cite{mann2012stick}, a lattice knot presentation is called \textit{irreducible} if there does not exist a lattice knot presentation with a smaller stick number and is called \textit{properly leveled} with respect to $w$ if each $w$-level of $\mathcal P$ contains a single connected polygonal arc.

A \textit{sh-lattice knot} is a non-intersecting polygon in the sh-lattice consisting of $x$-, $y$-, $z$-, and $w$-sticks. Similar to the corresponding definition in the cubic lattice, we define $|\mathcal P_x|$, $|\mathcal P_y|$, $|\mathcal P_z|$, and $|\mathcal P_w|$ to be the number of $x$-, $y$-, $z$-, and $w$-sticks used in a sh-lattice knot polygon $\mathcal P_{sh}$. The total number of sticks $|\mathcal P_{sh}|$ in a sh-lattice knot $\mathcal P_{sh}$ is therefore $|\mathcal P_{sh}| = |\mathcal P_x|+|\mathcal P_y|+|\mathcal P_z|+|\mathcal P_w|$. 

Throughout this paper, a polygon $\mathcal P$ always means a polygonal lattice knot, either in the cubic lattice or in the sh-lattice. Also, a cubic lattice (respectively, sh-lattice) polygon will always be assumed to be properly leveled with respect to $z$-sticks (respectively, $w$-sticks) unless otherwise specified.

\section{Linear Transformation between Lattices}\label{primary tool}
One tool we frequently apply in this paper is a linear transformation $T: \mathbb L^3\to \text{sh}$ defined by 
$$T\begin{bmatrix} x \\ y \\ z \end{bmatrix} = \begin{bmatrix} 1 & \frac{1}{2} & 0 \\ 0 & \frac{\sqrt{3}}{2} & 0 \\ 0 & 0 & 1 \end{bmatrix} \begin{bmatrix} x \\ y \\ z \end{bmatrix},$$ which sends the cubic lattice to the simple hexagonal lattice. As we will demonstrate in this section, this is a transformation between knot conformations in the cubic lattice and knot conformations in the simple hexagonal lattice that preserves many crucial properties. Therefore, $T$ allows us to consider results in the two lattices in an analogous way. The effect of $T$ is illustrated by \fullref{Teffect}. 

\begin{figure}[ht]
    \centering
    \subfigure[In cubic lattice, before applying $T$]{\includegraphics[width = .45\textwidth]{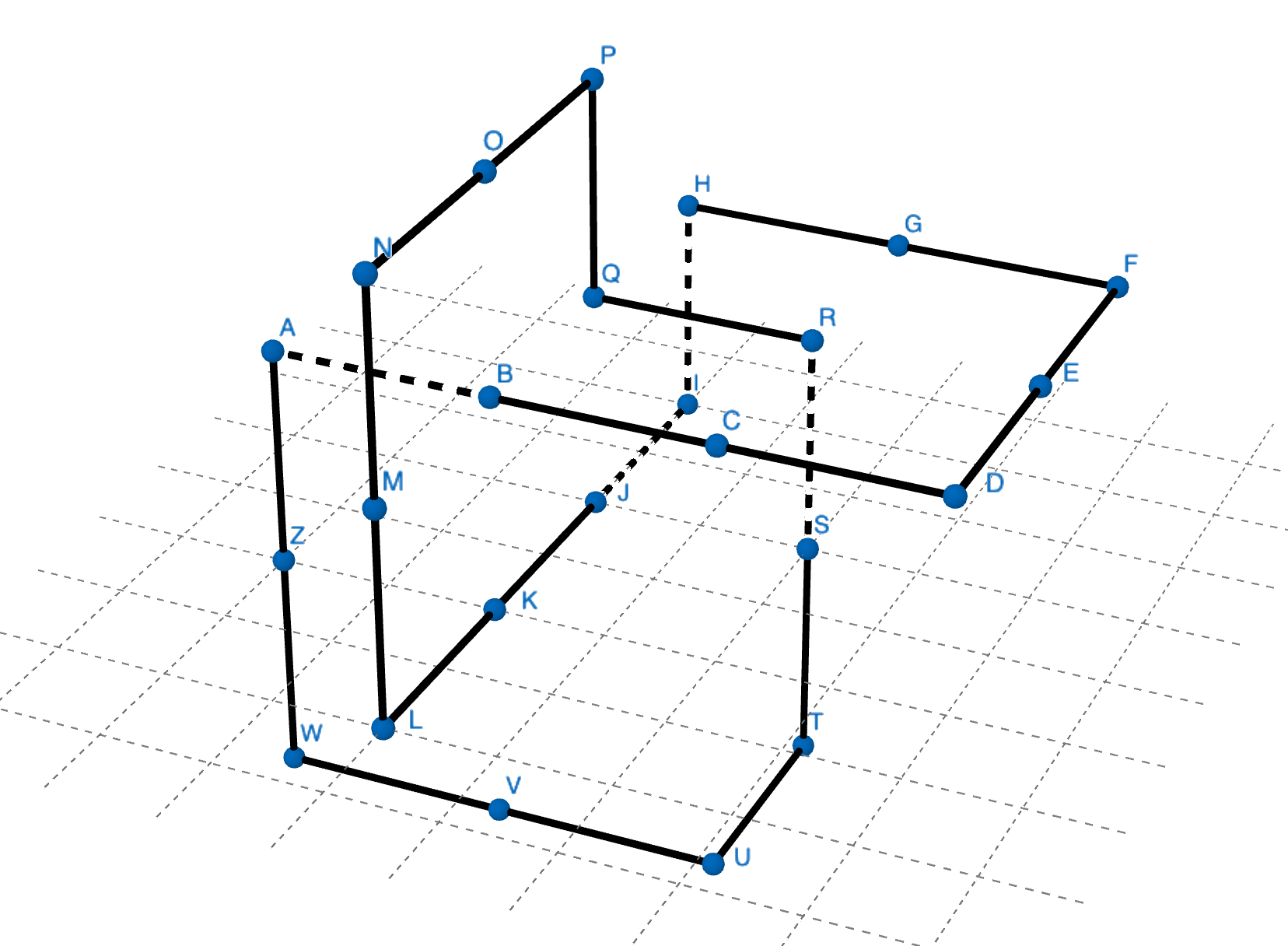}}
    \subfigure[In sh-lattice, after applying $T$]{\includegraphics[width = .45\textwidth]{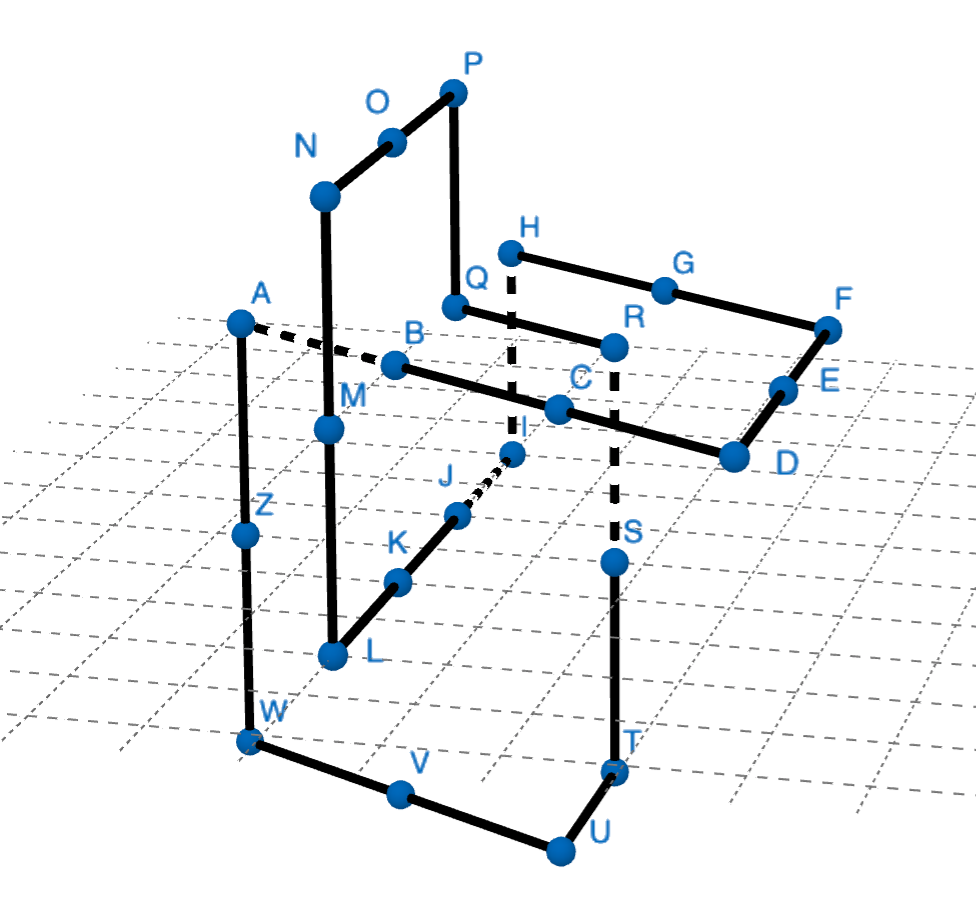}}
    \caption{Effect of $T$ on the trefoil knot}
    \label{Teffect}
\end{figure}

\begin{proposition}\label{prop1}
Let $\mathcal P_L$ be a polygon in the cubic lattice representing the knot type $[K]$. Applying the linear transformation $T = \begin{bmatrix} 1 & \frac{1}{2} & 0 \\ 0 & \frac{\sqrt{3}}{2} & 0 \\ 0 & 0 & 1 \end{bmatrix}$ produces a polygon $\mathcal P_{sh}$ in sh-lattice representing the same knot type.
\end{proposition}

\begin{proof}
Because $T$ is a linear transformation within a subset of $\R^3$ with $\det(T) = \frac{\sqrt{3}}{2}>0$, $T$ is invertible and orientation-preserving. Moreover, $T$ maps the basis vectors for the cubic lattice to those for the simple hexagonal lattice as $T \begin{bmatrix} 1 \\ 0 \\ 0 \end{bmatrix} = \begin{bmatrix} 1 \\ 0 \\ 0 \end{bmatrix}$, $T \begin{bmatrix} 0 \\ 1 \\ 0 \end{bmatrix} = \begin{bmatrix} \frac{1}{2} \\ \frac{\sqrt{3}}{2} \\ 0 \end{bmatrix}$, and $T \begin{bmatrix} 0 \\ 0 \\ 1 \end{bmatrix} = \begin{bmatrix} 0 \\ 0 \\ 1 \end{bmatrix}$. Finally, $T$ preserves knot type as it is an orientation-preserving homeomorphism.
\end{proof}

\begin{corollary} \label{cor2}
$T$ preserves stick number.
\end{corollary}

\begin{proof}
We can consider $|\mathcal P_{sh}|_x, |\mathcal P_{sh}|_y, |\mathcal P_{sh}|_w,$ and $|\mathcal P_{sh}|_z$ separately.

Beginning with $|\mathcal P_{sh}|_z$, note that as $T$ maps sticks in the cubic lattice only to $x$-, $y$-, or $w$-sticks in the simple hexagonal lattice, $|\mathcal P_{sh}|_z$ is necessarily 0.

Then, for $x$-, $y$-, and $w$- sticks, since $T$ maps $x$-sticks to $x$-sticks, $y$-sticks to $y$-sticks, and $z$-sticks to $w$-sticks from the cubic lattice to the sh-lattice, respectively $|\mathcal P_{sh}|_x = |\mathcal P_L|_x$, $|\mathcal P_{sh}|_y = |\mathcal P_L|_y$, and $|\mathcal P_{sh}|_w = |\mathcal P_L|_z$.

Where $|\mathcal P_L| = |\mathcal P_L|_x + |\mathcal P_L|_y + |\mathcal P_L|_z$ and $|\mathcal P_{sh}| = |\mathcal P_{sh}|_x + |\mathcal P_{sh}|_y + |\mathcal P_{sh}|_z + |\mathcal P_{sh}|_w$, we have that 
\begin{align*}
|\mathcal P_L| &= |\mathcal P_L|_x + |\mathcal P_L|_y + |\mathcal P_L|_z \\
&= |\mathcal P_{sh}|_x + |\mathcal P_{sh}|_y + 0 + |\mathcal P_{sh}|_w \\
&= |\mathcal P_{sh}|_x + |\mathcal P_{sh}|_y + |\mathcal P_{sh}|_z + |\mathcal P_{sh}|_w \\
&= |\mathcal P_{sh}|
\end{align*}
as desired.
\end{proof}

\begin{lemma} \label{lemma1}
 $T$ preserves the order and length of sticks. That is, $\mathcal P_L$ and $\mathcal P_{sh} = T(\mathcal P_L)$ can be represented by the same series of sticks of the same lengths, where $\mathcal P_L$ has sticks in the cubic lattice and $\mathcal P_{sh}$ sticks in the sh-lattice.
\end{lemma}

\begin{proof}
Consider the sequence of sticks forming $\mathcal P_L$. Without loss of generality, we denote the sequence to be \begin{center} $\mathcal P_L = x^{a_1} y^{b_1} \ldots x^{a_i} y^{b_i} z^{c_1} x^{a_{i + 1}} y^{b_{i + 1}} \ldots x^{a_j} y^{b_j} z^{c_2} \ldots x^{a_n} y^{b_n} z^{c_m}$ \end{center} where $a_1, \ldots a_n, b_1, \ldots b_n, c_1, \ldots c_m \in \mathbb{Z}$. The magnitude of an exponent gives us the length of its corresponding stick, and the sign of an exponent gives the stick's direction. For example, $x^{-1}$ could also be written as $X$ according to the notation proposed in \cite{bailey2015stick}. Moreover, We can rewrite the expression of $\mathcal P_L$ as $a_1 x + b_1 y + \ldots + a_i x + b_i y + c_1 z + a_{i + 1} x + b_{i + 1} y + \ldots + a_j x + b_j y + c_2 z + \ldots + a_n x + b_n y + c_m z$, where $x$, $y$, and $z$ are now expressed as vectors. Applying $T$ produces
	\begin{align*}
	    T(\mathcal P_L) &= T(x^{a_1} y^{b_1} \cdots x^{a_i} y^{b_i} w^{c_1} x^{a_{i + 1}} y^{b_{i + 1}} \cdots x^{a_j} y^{b_j} w^{c_2} \cdots x^{a_n} y^{b_n} w^{c_m}) \\
	    &= T(a_1 x + b_1 y + \ldots + a_i x + b_i y + c_1 w + a_{i + 1} x + b_{i + 1} y + \ldots + a_j x + \\ &b_j y + c_2 w + \ldots + a_n x + b_n y + c_m w) \\
	    &= a_1 Tx + b_1 Ty + \ldots + a_i Tx + b_i Ty + c_1 Tw + a_{i + 1} Tx + b_{i + 1} Ty + \ldots + a_j Tx + \\ &b_j Ty + c_2 Tw + \ldots + a_n Tx + b_n Ty + c_m Tw \\
	    &= a_1 x_{sh} + b_1 y_{sh} + \ldots + a_i x_{sh} + b_i y_{sh} + c_1 w_{sh} + a_{i + 1} x_{sh} + b_{i + 1} y_{sh} + \ldots + a_j x_{sh} + \\ &b_j y_{sh} + c_2 w_{sh} + \ldots + a_n x_{sh} + b_n y_{sh} + c_m w_{sh} \\
	    &= x_{sh}^{a_1} y_{sh}^{b_1} \cdots x_{sh}^{a_i} y_{sh}^{b_i} w_{sh}^{c_1} x_{sh}^{a_{i + 1}} y_{sh}^{b_{i + 1}} \cdots x_{sh}^{a_j} y_{sh}^{b_j} w_{sh}^{c_2} \cdots x_{sh}^{a_n} y_{sh}^{b_n} w_{sh}^{c_m}
	\end{align*}

It is clear that the result of these manipulations has the same order and length of sticks.
\end{proof}

\section{Upper Bound of Stick Number and Edge Length in sh-lattice}\label{upper bound results}

\begin{definition} [Stick Number of a Knot Type]
The \textit{stick number of a knot type} is the smallest number of sticks necessary to build a knot in $\mathbb{R}^3$. With respect to a lattice $\mathcal{L}$, we define the stick number $s_\mathcal{L}$ to be the smallest number of sticks among all knot conformations $\mathcal P$ of $[K]$ in $\mathcal{L}$, i.e. $s_\mathcal{L}[K] = \min\limits_{\mathcal P\in [K], \mathcal P \subset \mathcal L} |\mathcal P|$. In particular, we denote the stick number with respect to the cubic lattice as $s_L[K] = \min\limits_{\mathcal P\in [K],\mathcal P\subset \mathbb L^3} |\mathcal P|$ and denote the stick number with respect to the sh-lattice as $s_{sh}[K] = \min\limits_{\mathcal P\in [K],\mathcal P\subset sh} |\mathcal P|$. 
The \textit{stick number of a knot type} $[K]$ is the least stick number among all knot conformations $\mathcal P$ of $[K]$ in a given lattice, i.e. $s[K] = \min\limits_{\mathcal P\in [K]} |\mathcal P|$. In particular, we denote the stick number with respect to the cubic lattice as $s_L[K] = \min\limits_{\mathcal P\in [K],\mathcal P\subset \mathbb L^3} |\mathcal P|$ and denote the stick number with respect to the sh-lattice as $s_{sh}[K] = \min\limits_{\mathcal P\in [K],\mathcal P\subset sh} |\mathcal P|$. 
\end{definition}

\begin{proposition}\label{prop4.2}
For any knot type $[K]$, $s_{sh}[K] \leq s_{L}[K]$, where $s_{sh}$ is the stick number of $[K]$ in the simple hexagonal lattice and $s_{L}$ is the stick number of $[K]$ in the cubic lattice.
\end{proposition}

\begin{proof}
Suppose the polygon $\mathcal P_L$ representing knot type $[K]$ in the cubic lattice has minimal stick number, i.e. $|\mathcal P_L| = s_L[K]$. Now consider $T(\mathcal P_L) = \mathcal P_{sh}$. If $\mathcal P_{sh}$ also has minimal stick number, that is $|\mathcal P_{sh}| = s_{sh}[K]$, then $s_{sh}[K] = s_L[K]$. If $\mathcal P_{sh}$ does not have minimal stick number, then $s_{sh}[K] < s_L[K]$. Taken together, $s_{sh}[K] \leq s_L[K]$.
\end{proof}

\begin{remark}
\fullref{prop4.2} can be proved in a different way: from \fullref{prop1}, we have that for any polygon $\mathcal P_L$ in the cubic lattice, there exists a polygon $\mathcal P_{sh} = T(\mathcal P_L)$ in the simple hexagonal lattice representing the same knot type and with the same number of sticks. Where $\mathcal P_L$ is the presentation of a knot type $[K]$ with $s_L[K]$ sticks, the minimal presentation of $[K]$ in the simple hexagonal lattice has at most $s_L[K]$ sticks. Hence, $s_{sh}[K] \leq s_L[K]$.
\end{remark}

It is reasonable to conjecture that the bound presented in \fullref{prop4.2} should be strict. Note that $T$ preserves the stick number as well as the overall structure of the knot because of the properties in \fullref{primary tool}, but the $z$-sticks in sh-lattice are left unused since they are not in the image of $T$. Therefore, as demonstrated by \fullref{yx-z} below, one can imagine finding a corner consisting of an $x$-stick and a $y$-stick, and reducing the corner to a $z$-stick instead. This is possible via ambient isotopy, as long as there are no other sticks in the way of the potential trajectory. 

\begin{figure}[ht]
    \centering
    \begin{tikzpicture}
    \draw[line width=0.4mm] (0,0) -- ++(1,{sqrt(3)});
    \draw[line width=0.4mm] (0,0) -- ++(2,0);
    \draw[line width=0.4mm, ->] (3,1) --++(2,0);
    \draw[line width=0.4mm] (6,{sqrt(3)}) -- ++(1,{-sqrt(3)});
    \end{tikzpicture}
    \caption{Transforming an $xy$-corner to a $z$-stick}
    \label{yx-z}
\end{figure}
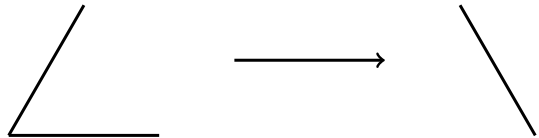

\begin{lemma} \label{lemma2}
Project a polygon $\mathcal P$ in the cubic lattice down to the $xy$-plane. Suppose we have an $x$-stick and a $y$-stick of equal length connected in the shape of an ``L". We will call the $x$-stick $x$ and the $y$-stick $y$. If there are no $z$-sticks within the triangle with $x$ and $y$ as legs, then we can replace them with a $z$-stick in the hexagonal lattice after applying $T$. 
\end{lemma}

\begin{proof}
Suppose we have an intersection as above. Label the endpoints of the sticks as $(0,0), (c,0), (0,c)$ where $c$ is the length of the legs. Then, applying $T$, we get new endpoints at $(0,0), (c,0), (\frac{c}{2},\frac{c\sqrt{3}}{2})$. Thus, we can connect $(\frac{c}{2},\frac{c\sqrt{3}}{2})$ and $(0,0)$ with a $z$-stick $\langle\frac{-c}{2},\frac{c\sqrt{3}}{2}\rangle$ replacing our initial $x$ and $y$-sticks and reducing the stick number of $\mathcal P$ by 1. This does not change the knot type because there are no $z$-sticks in the triangle so we can smoothly deform the $x$ and $y$-sticks into the $z$-stick. 
\end{proof}

\begin{theorem}\label{conjecture2}
For any knot type $[K]$, $s_{sh}[K] < s_{L}[K]$. 
\end{theorem}

\begin{proof}
Consider some intersection of an $x$-stick and a $y$-stick in the cubic lattice. Also, rotate the knot so that this intersection is in the shape of an ``L". Let us label $a$ as the $x$-stick and $b$ as the $y$-stick. Without loss of generality, suppose the length of $a$ is less than or equal to that of $b$. Then, consider the $x$-level that $b$ sits on. Extend all the $x$-sticks to the right of this $x$-level by the difference in lengths between $a$ and $b$. Now, by connecting the farthest points on $a$ and $b$, we can visualize an isosceles right triangle when we project the knot down onto the $xy$-plane. Note that we are not adding a stick here. 
 
 If there are no other sticks in this triangle, we can apply \fullref{lemma2} and reduce the stick number by 1. Otherwise, suppose that there are $n$ $z$-sticks in the triangle or endpoints of $x$- or $y$-sticks in the triangle, which we will label as $p_1, \ldots, p_n$. Let the intersection of $x$ and $y$ be the origin and let $p_x$ be the $x$-coordinate of the leftmost $p_i$. Also, let $b_y$ be the $y$-coordinate of the upper endpoint of $b$. Note that the line $y=b_y-x$ forms the hypotenuse of a right triangle connecting the endpoints of $a$ and $b$. Thus if the sum of $x$- and $y$-coordinates of a point is greater than $b_y$, then it will lie outside the triangle. 

Our goal here is to squeeze the region strictly bounded by $y$-levels $b_y$ and $0$ into the region strictly bounded by $y$-levels $b_y$ and $b_y-p_x$. Doing so will cause all $z$-sticks or endpoints of $x$- and $y$-sticks to have $x$-coordinate greater than or equal to $p_x$ and $y$-coordinate greater than $b_y-p_x$ so the sum of their coordinates will be greater than $b_y$. So, let us look at the endpoints of all sticks in the knot. Let $y$ be the current $y$-coordinate of the points and $y'$ the new coordinate according to the following equation:

  \[
y' = 
\begin{cases} 
b_y-(b_y-y)(\frac{b_y-(b_y-p_x)}{b_y-0}) & \text{if $ 0 < y < b_y$} \\
y & \text{otherwise} \\
\end{cases}
\]
  which simplifies to
  \[
y' = 
\begin{cases} 
b_y-p_x+\frac{p_x}{b_y}y & \text{if $ 0 < y < b_y$} \\
y & \text{otherwise} \\
\end{cases}
\]

 This scales the distance between the sticks and the $b_y$ level so that they lie within the desired region, as demonstrated in \fullref{thm4.4} below.

\begin{figure} [ht]
    \centering
    \begin{tikzpicture}[scale = 0.9]
    \draw[line width = 0.4mm] (0,0) --++(3,0);
    \draw[line width = 0.4mm] (0,0) --++(0,5);
    \fill[red!100] (1,1) circle(.12);
    \fill[red!100] (1.5,0.5) circle(.12);
    \fill[red!100] (2,1.5) circle(.12);
    \fill[red!100] (2.5,0.2) circle(.12);
    \draw[line width = 0.5mm, ->] (3.5,3) -- ++ (1,0);
    \fill[yellow!30] (5,0) --++(5,0) --++(0,5)--++(-5,0)--++(0,-5);
    \draw[line width = 0.4mm] (5,0) --++(0,5);
    \draw[line width = 0.4mm, red, dotted] (5,0) --++(2,0);
    \draw[line width = 0.4mm] (7,0) --++(3,0);
    \draw[black!40] (5,5) --++(5,-5);
    \fill[red!100] (8,1) circle(.12);
    \fill[red!100] (8.5,0.5) circle(.12);
    \fill[red!100] (9,1.5) circle(.12);
    \fill[red!100] (9.5,0.2) circle(.12);
    \draw[line width = 0.5mm, ->] (10.5,3) -- ++ (1,0);
    \fill[yellow!30] (12,2.5) --++(5,0) --++(0,2.5)--++(-5,0)--++(0,-2.5);
    \draw[line width = 0.4mm] (12,0) --++(0,5);
    \draw[line width = 0.4mm] (12,0) --++(5,0);
    \draw[black!40] (12,5) --++(5,-5);
    \fill[red!100] (15,4.2) circle(.12);
    \fill[red!100] (15.5,4.1) circle(.12);
    \fill[red!100] (16,4.3) circle(.12);
    \fill[red!100] (16.5,4.04) circle(.12);
    \node at (-0.5,-0.5) {$(0,0)$};
    \node at (-0.5, 2.5) {$b$};
    \node at (-0.5, 5.5) {$b_y$};
    \node at (2,-0.5) {$a$};
    \node at (6,-0.5) {$b_y-\text{length}(a)$};
    \node at (9,-0.5) {$a$};
    \node at (4.3,0) {$(0,0)$};
    \node at (4.5, 5.5) {$b_y$};
    \node at (11.5, -0.5) {$(0,0)$};
    \node at (11.5, 5.5) {$b_y$};
    \node at (11.1, 2.3) {$b_y-p_x$};
    \end{tikzpicture}
    \caption{Illustration of scaling}
    \label{thm4.4}
\end{figure}
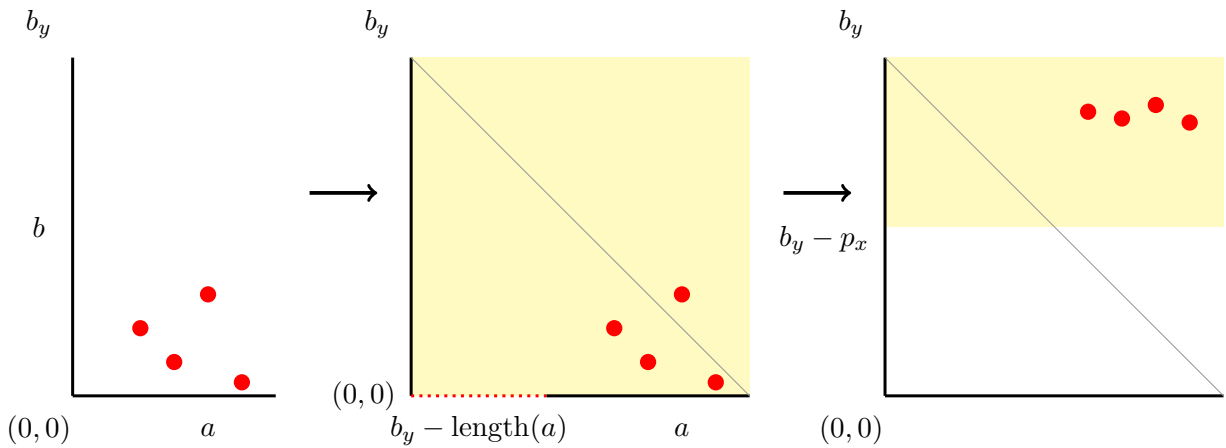
 
The value we want to scale is $b_y-y$ because it is the distance between the specified point and $b_y$. Then, we want the scaling factor to be the ratio between the total length between the $y$-levels $b_y$ and $0$ and the length between the levels $b_y$ and $b_y-p_x$. Therefore, we make the denominator $b_y-0$ and the numerator $b_y-(b_y-p_x)$. Putting this together, we take $b_y$ and subtract the new distance to $b_y$ which gives us the above equation. 

  Thus, the whole section of the knot lying below the hypotenuse now lies above it so all the $z$-sticks are out of our triangle. 
 This new diagram does not necessarily lie in the cubic lattice because we may have non-integer coordinates. Scale everything in the $x$- and $y$-direction by $b_y$. 
 
 Now all our sticks lie back in the lattice. We can apply $T$ and reduce by \fullref{lemma2}. Thus, we have reduced the stick number by 1, proving the theorem. 
\end{proof}

\begin{remark}
Combining the proposition above with Theorem $1$ of \cite{bailey2015stick} yields $5b[K]\le s_{sh}[K] < s_L[K]$, where $b[K]$ is the bridge number of knot type $[K]$. 
\end{remark}
 
 \begin{definition}[Edge, Edge Length]
An \textit{edge} of a polygon in a lattice is a unit-length segment of the polygon between two points in the lattice. The \textit{edge length} of a polygon in a lattice is the total number of edges in the polygon. 

For polygon $\mathcal{P} \in [K]$, we denote $\mathcal E_L(\mathcal{P})$ to be the edge length of a polygon $\mathcal P$ in the cubic lattice, and $\mathcal E_{sh}(\mathcal{P})$ to be the edge length of a polygon $\mathcal P$ in the sh-lattice.
In the lattice $\mathcal{L}$ and direction $\alpha$, the total number of edges of the $\alpha$-sticks of polygon $\mathcal P$ in the lattice $\mathcal{L}$ is denoted as $\mathcal{E}_{\mathcal{L};\alpha}(\mathcal P)$.
Furthermore, $e_L[K]$ and $e_{sh}[K]$ are the minimal edge lengths of a knot type $[K]$ in the cubic and sh-lattices, respectively.
\end{definition}

We conclude this section by proving an analogous strict bound on edge length. 

\begin{proposition}
$T$ preserves edge length.
\end{proposition}

\begin{proof}
Let $\mathcal P_L$ be a knot in the cubic lattice. Then $\mathcal P_{sh} = T(\mathcal P_L)$ has the same knot type but is in the sh-lattice, according to \fullref{prop1}. It now suffices to show $\mathcal{E}_L(\mathcal P_L) = \mathcal{E}_{sh}(\mathcal P_{sh})$, which can be done by examining the number of edges used in the creation of each kind of sticks.

Since no sticks in $\mathcal P_L$, and in particular no edges in $\mathcal P_L$, are mapped to $z$-sticks in the sh-lattice by $T$, $\mathcal{E}_{sh;z}(\mathcal P_{sh}) = 0$. Then, as in \fullref{cor2}, $T$ maps $x$-sticks to $x$-sticks, $y$-sticks to $y$-sticks, and $z$-sticks to $w$-sticks. By \fullref{lemma1}, $T$ preserves stick length, so for a given stick $a$, the number of edges to form $a$ in the cubic lattice and, after applying $T$, in the sh-lattice are equal. It follows that $\mathcal{E}_{L;x}(\mathcal P_L) = \mathcal{E}_{sh;x}(\mathcal P_{sh})$, $\mathcal{E}_{L;y}(\mathcal P_L) = \mathcal{E}_{sh;y}(\mathcal P_{sh})$, and $\mathcal{E}_{L;z}(\mathcal P_L) = \mathcal{E}_{sh;w}(\mathcal P_{sh})$. Therefore, we have 
\begingroup\allowdisplaybreaks
\begin{align*}
    \mathcal{E}_L(\mathcal P_L) &= \mathcal{E}_{L;x}(\mathcal P_L) + \mathcal{E}_{L;y}(\mathcal P_L) + \mathcal{E}_{L;z}(\mathcal P_L) \\
    &= \mathcal{E}_{sh;x}(\mathcal P_{sh}) + \mathcal{E}_{sh;y}(\mathcal P_{sh}) + \mathcal{E}_{sh;w}(\mathcal P_{sh}) + 0 \\
    &= \mathcal{E}_{sh;x}(\mathcal P_{sh}) + \mathcal{E}_{sh;y}(\mathcal P_{sh}) + \mathcal{E}_{sh;w}(\mathcal P_{sh}) + \mathcal{E}_{sh;z}(\mathcal P_{sh}) \\
    &= \mathcal{E}_{sh}(\mathcal P_{sh})
\end{align*}
\endgroup
as desired.
\end{proof}
 
\begin{theorem}
For any knot type $[K]$, $e_{sh}[K] < e_{L}[K]$, where $e_{sh}$ is the edge length of $[K]$ in the simple hexagonal lattice and $e_{L}$ is the edge length of $[K]$ in the cubic lattice. 
\end{theorem}

\begin{proof}
\fullref{reduceedge} shows the idea of the following proof. Let $\mathcal P$ be a knot polygon in the cubic lattice with minimal edge length. Pick any arbitrary corner $v$ of $\mathcal P$. Note that by rotating the whole knot, we do not change the knot type. Rotate this knot so that the sticks that intersect at $v$ are a $y$- and an $x$-stick in the shape of an ``L" when projected down onto the $xy$-plane. Now, label the coordinates of $v$ to be $(v_x,v_y)$. We can create an isosceles right triangle with the vertices $(v_x,v_y)$, $(v_x+1,v_y)$, $(v_x,v_y+1)$. Because we are in the cubic lattice and this isosceles right triangle has side length $1$, no $z$-sticks can occur within the region of this triangle. Thus, by \fullref{lemma2}, when we apply $T$, we can replace the legs of this triangle with a $z$-stick in the sh-lattice, which reduces the edge length by $1$. 
 \begin{figure}[ht]
\includegraphics[width = \textwidth]{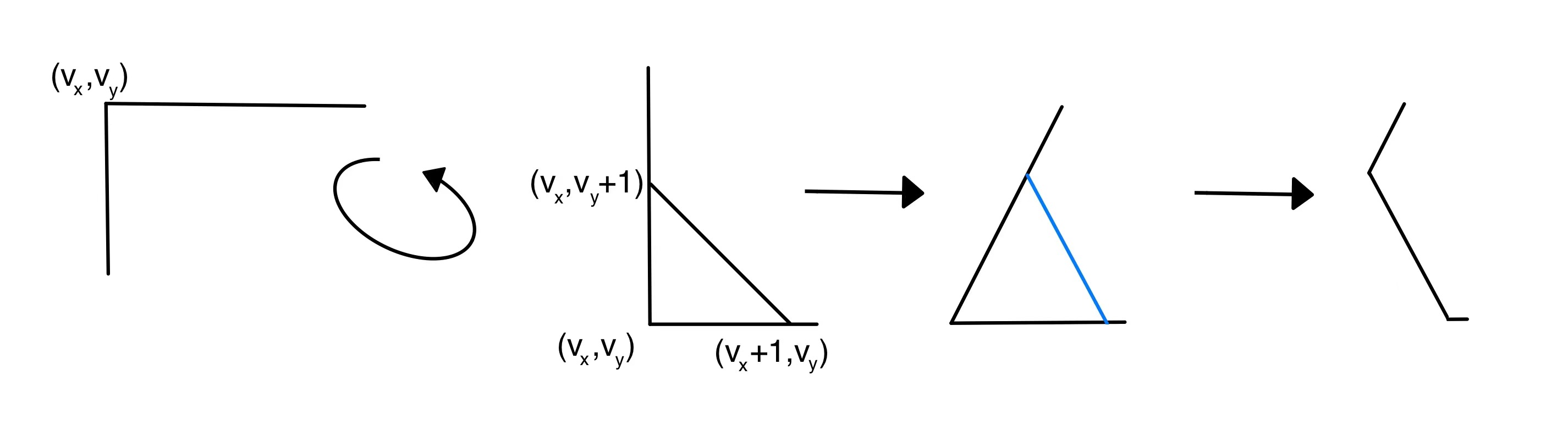}
\caption{Reducing edge length by 1}
\label{reduceedge}
\end{figure}
\end{proof}

Similar to the upper bound, we can find a lower bound on edge length in the sh-lattice for a knot type $[K]$ with respect to the edge length in the cubic lattice. 

\begin{proposition}\label{4.10}
For a non-trivial knot type $[K]$, $e_{sh}[K] \geq \frac{3e_L[K]+30}{8}$. 
\end{proposition}
\begin{proof}
To represent the knot type $[K]$, we choose polygon $\mathcal{P}$ that has $e_{sh}[K]$ edges. We aim to, while preserving $w$-sticks, transform $w$-levels in the sh-lattice into $z$-levels in the cubic lattice. Before we begin, note that we can rotate our polygon around the vertical axis so that the following is true: 
\[ \mathcal{E}_{sh;x}(\mathcal{P})\ge\frac{1}{3}(e_{sh}(\mathcal{P})-\mathcal{E}_{sh;w}(\mathcal{P})). \]

Now to form each $w$-level in the sh-lattice into a $z$-level in the cubic lattice, we replace each $x$-edge in the sh-lattice by $x^2$ in the cubic lattice, each $y$-edge in the sh-lattice by $xy^2$, and each $z$-edge by $y^2 x^{-1}$. Wherever the results of our substitutions may overlap, simply remove the overlapping edge so there is exactly one edge in its place. With this in mind, we have the following inequalities:
\begin{align*}
e_{L}[K] & \leq \mathcal{E}_L (\mathcal{P}) \\
& \leq \mathcal{E}_{sh;w}(\mathcal{P}) + 2 \mathcal{E}_{sh;x}(\mathcal{P}) + 3 (\mathcal{E}_{sh;y}(\mathcal{P}) + \mathcal{E}_{sh;z}(\mathcal{P})) \\
& = 3 e_{sh}(\mathcal{P}) - 2 \mathcal{E}_{sh;w}(\mathcal{P}) - \mathcal{E}_{sh;x}(\mathcal{P}) \\
& \leq 3 e_{sh}(\mathcal{P}) - 2 \mathcal{E}_{sh;w}(\mathcal{P}) - \frac{1}{3} (e_{sh}(\mathcal{P}) - \mathcal{E}_{sh;w}(\mathcal{P})) \\
& = \frac{8}{3} e_{sh}(\mathcal{P}) - \frac{5}{3} \mathcal{E}_{sh;w}(\mathcal{P}) \\
& \leq \frac{8}{3} e_{sh}(\mathcal{P}) - \frac{5}{3} \times 6 \\
& = \frac{8}{3} e_{sh}(\mathcal{P}) - 10
\end{align*}

Note that $\mathcal{E}_{sh;w}(\mathcal{P}) \geq 6$ by Lemma 1.4 from \cite{mann2012stick}, which establishes that since $\mathcal{P}$ is non-trivial, $|\mathcal{P}|_w \geq 4$, and by Lemma 1.3 from \cite{mann2012stick}, which tells us that $w$-sticks with an endpoint on a boundary level must have at least two edges. By manipulating the final inequality, we find the desired result.
\end{proof}

\section{Classification of 11-stick Knots}\label{specific knots}

In this section, we classify the possible knot types and explore related properties of knots with stick number $11$ in the simple hexagonal lattice. Unless specified otherwise, in what follows all discussion related to the knots, e.g. the stick numbers and other concepts, is in the context of simple hexagonal lattice, and all polygons have eleven sticks.

\begin{proposition}\label{4-1knot}
The stick number of a figure-eight knot in the sh-lattice is $11$, i.e. $s_{sh}(4_1) = 11$. 
\end{proposition}

\begin{proof}
\fullref{fig:4_1knot11} is a knot presentation of $4_1$ in sh-lattice with eleven sticks. The $w$-sticks are colored black, and for others, sticks in the same plane parallel to the $xy$-plane are marked with the same color.
Moreover, by Theorem 3 in \cite{bailey2015stick}, any non-trivial knot in the simple hexagonal lattice should have at least eleven sticks. Therefore, $s_{sh}(4_1) = 11$. 
\begin{figure}[ht]
    \centering
    \includegraphics[scale = 0.4]{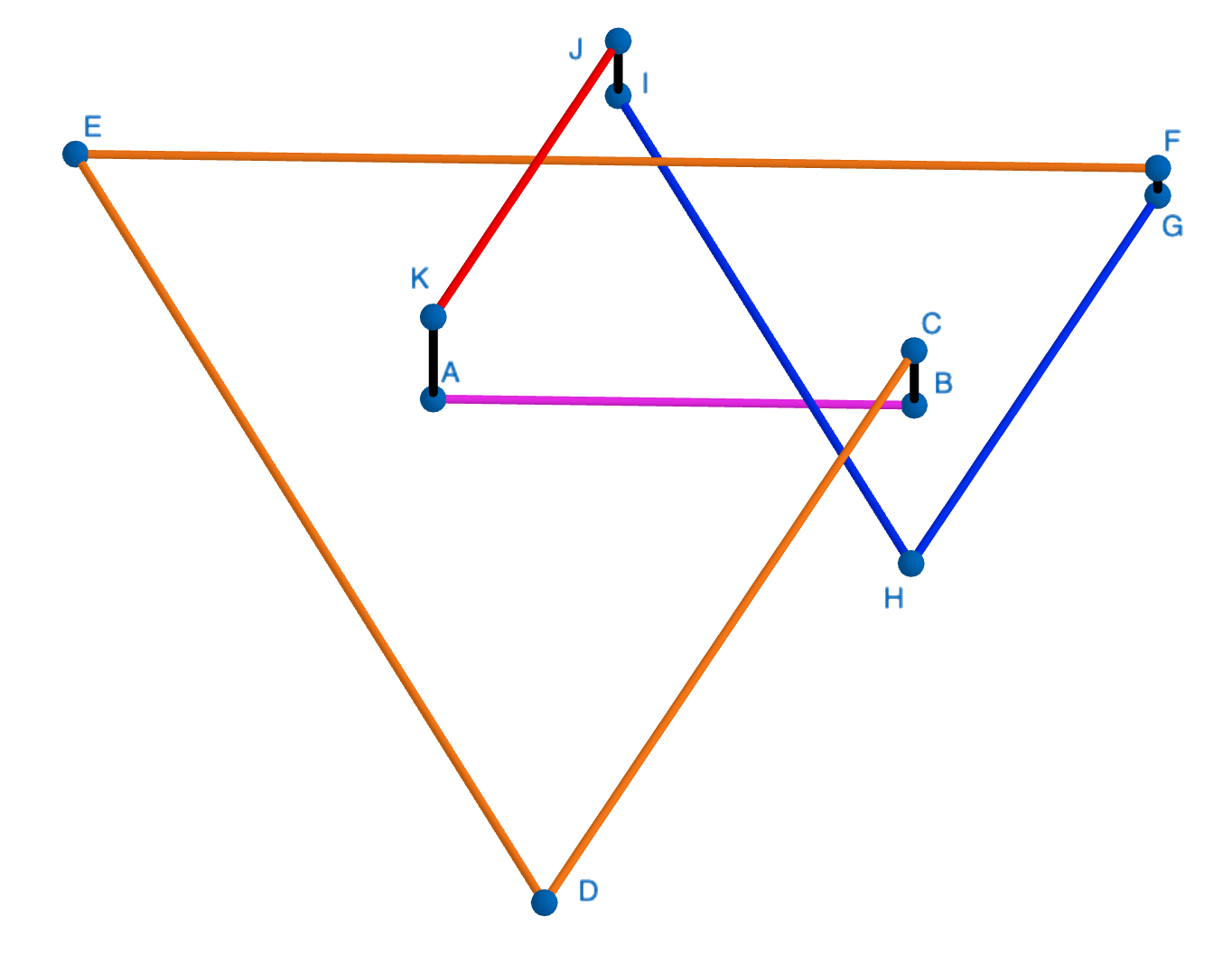}
    \caption{$4_1$ knot in sh-lattice with 11 sticks}
    \label{fig:4_1knot11}
\end{figure}
\end{proof}

Recall that we have the following results in the cubic lattice:

\begin{theorem*}[Huang \& Yang, 2017, \cite{huang2017lattice}]
The only non-trivial knot types $[K]$ with stick number $s_L[K] \le 15$ in the cubic lattice are $3_1$ and $4_1$.
\end{theorem*}

By \cite{bailey2015stick} and \cite{mann2012stick}, we know the trefoil knot $3_1$ has stick number $11$ in the sh-lattice. Combined with \fullref{4-1knot} above, we will prove a result analogous to the above theorem but in the sh-lattice. That is, the only non-trivial knot types $[K]$ with stick number $s_{sh}[K] \le 11$ in the sh-lattice are $3_1$ and $4_1$. In pursuit of this result, we begin by determining the number of $w$-sticks required in an 11-stick polygon.

\begin{proposition}\label{stickequallevel}
Suppose that $\mathcal P$ is any polygon properly leveled with respect to $w$, then the number of $w$-sticks is equal to the number of $w$-levels.
\end{proposition}

\begin{proof}
Suppose we have $n$ $w$-sticks in $\mathcal P$. Note that $\mathcal P$ cannot have more than $n$ $w$-levels, as if it did, the polygon would not be closed or would be reducible. Suppose for the sake of contradiction that $\mathcal P$ has fewer than $n$ $w$-levels.

We know that each $w$-stick has two endpoints. In other words, each $w$-stick has an endpoint on two distinct $w$-levels. Moreover, a given $w$-level cannot have fewer than two endpoints of $w$-sticks, since such a level would be disconnected from the rest of the knot. Thus, as there are $2n$ $w$-stick endpoints and each $w$-level must have at least two endpoints, there must be at most $n$ $w$-levels with at least two $w$-stick endpoints a piece.

Suppose there exists some $w$-level with an odd number of endpoints. Then the polygon is no longer a knot since the arc of the knot on that $w$-level must cause the knot to branch at one of the endpoints or not pass through an endpoint, making it no longer part of the rest of the knot. If there are an even composite number of endpoints, connected in pairs by disjoint arcs, then the knot is not properly leveled. Hence, there cannot be $w$-levels with more than two endpoints, and the result is true.
\end{proof}

\begin{remark}
    Note that if a properly leveled 11-stick polygon has more than five $w$-sticks, then there is an insufficient number of non-$w$-sticks to connect the $w$-sticks (there must be at least distinct six $w$-levels, but only at most five non-$w$-sticks), so a properly leveled 11-stick polygon can have at most five $w$-sticks. 
\end{remark}

Before continuing, we introduce a useful tool, as defined in \cite{huh2005lattice}.

\begin{definition} [$R$-move] \label{R-move}
Let $s$ and $t$ be connected, perpendicular sticks in a polygon $\mathcal{P}$. Also, let $R$ be the rectangle with $s$ and $t$ as sides, where $s'$ and $t'$ denote the sides opposite to $s$ and $t$ in $R$ respectively. Note that $R$ is unique. If $R$ does not intersect with any other part of $\mathcal{P}$, exchange $s$ and $t$ of $\mathcal{P}$ with $s'$ and $t'$ to get a new polygon $\mathcal{P'}$. Since $R$ does not intersect with any other part of $\mathcal{P}$, $\mathcal{P}'$ is equivalent to $\mathcal{P}$. This operation is called an \textit{$R$-move} of $s$ and $t$. Note that we will also refer to this as an $R$-move on the rectangle $R$, with $s$ and $t$ the sides of $R$ in the polygon $\mathcal{P}$ before the $R$-move.
\end{definition}

\begin{corollary}\label{specificsticks}
A non-trivial irreducible $11$-stick polygon $\mathcal P$ with four $w$-sticks has $w$-sticks exactly $w_{13}$, $w_{23}$, $w_{24}$, and $w_{14}$. 
\end{corollary}

\begin{proof}
Lemma $1.2$ in \cite{mann2012stick} says that there are no two $w$-sticks in $\mathcal P$ with their endpoints on the same $w$-levels. Because $\mathcal{P}$ has four $w$-sticks, there are only three possible combinations of the $w$-sticks in $\mathcal P$:
\begin{enumerate} [label = (\alph*)]
    \item $w_{12}$, $w_{23}$, $w_{34}$, $w_{14}$.
    
    \item $w_{12}$, $w_{24}$, $w_{34}$, $w_{13}$, 
    
    \item $w_{13}$, $w_{23}$, $w_{24}$, $w_{14}$.
\end{enumerate}
By Lemma $1.3$ in \cite{mann2012stick}, both $(a)$ and $(b)$ are the unknot. Therefore, every non-trivial irreducible polygon $\mathcal P$ with eleven sticks and four $w$-sticks has the said configuration. 
\end{proof}

\begin{lemma}
\label{atleastonelemma}
A non-trivial $11$-stick polygon $\mathcal{P}$ satisfies $\min\{|\mathcal P|_x,|\mathcal P|_y,|\mathcal P|_z\}\ge 1$. 
\end{lemma}

\begin{proof}
If $\mathcal{P}$ has no $z$-sticks, then this implies there exists a non-trivial polygon in the cubic lattice with eleven sticks. (In particular, it is the preimage of $\mathcal{P}$ with respect to linear transformation $T$.) This is false by results in \cite{huh2005lattice}. Similarly, if $\mathcal{P}$ has no $x$-sticks or no $y$-sticks, we rotate the polygon so that there are no $z$-sticks in the knot instead.
\end{proof}

\begin{lemma}
A non-trivial $11$-stick polygon $\mathcal{P}$ with four $w$-sticks must have at least two sticks for two of $x$-sticks, $y$-sticks, and $z$-sticks and at least one stick for the remaining stick type. 
\end{lemma}

\begin{proof}
Since we have four $w$-sticks, we know $|\mathcal P|_x+|\mathcal P|_y+|\mathcal P|_z = 7$. Suppose, towards contradiction, that the proposition is false, then either only two of the three types of sticks are used, or exactly two types of stick have stick number $1$. The first case is impossible in \fullref{atleastonelemma}. We now prove the second case to be impossible as well. 

Without loss of generality, we assume that $\mathcal P$ has five $x$-sticks. Recall that $\mathcal P$ should have four $w$-levels. Observe that exactly one of the $w$-levels has two $x$-sticks, therefore requiring the presence of a $y$- or $z$-stick on the level. Suppose the level has a sequence $xyzx$ or $xzyx$, then because we only have two $w$-sticks not attached to this level, there is always a way for us to reduce the stick number of polygon to $10$, resulting in the unknot. Therefore, we can suppose the level has an $xyx$ or $xzx$ sequence. We assume that the polygon has the $xyx$ sequence, and the other case would follow with the same proof. Note that this particular $w$-level should not be on the boundary, otherwise we can reduce the stick number by $1$. We hereby assume the sequence is on $w$-level $2$ (the proof for the $w$-level $3$ case is the same). Now, the view on level $2$ looks like \fullref{xyx1} as the $xyx$ sequence passes between $w_{13}$ and $w_{14}$, preventing the stick number from being reduced. 
\begin{figure}[ht]
    \centering
    \includegraphics[height=5cm]{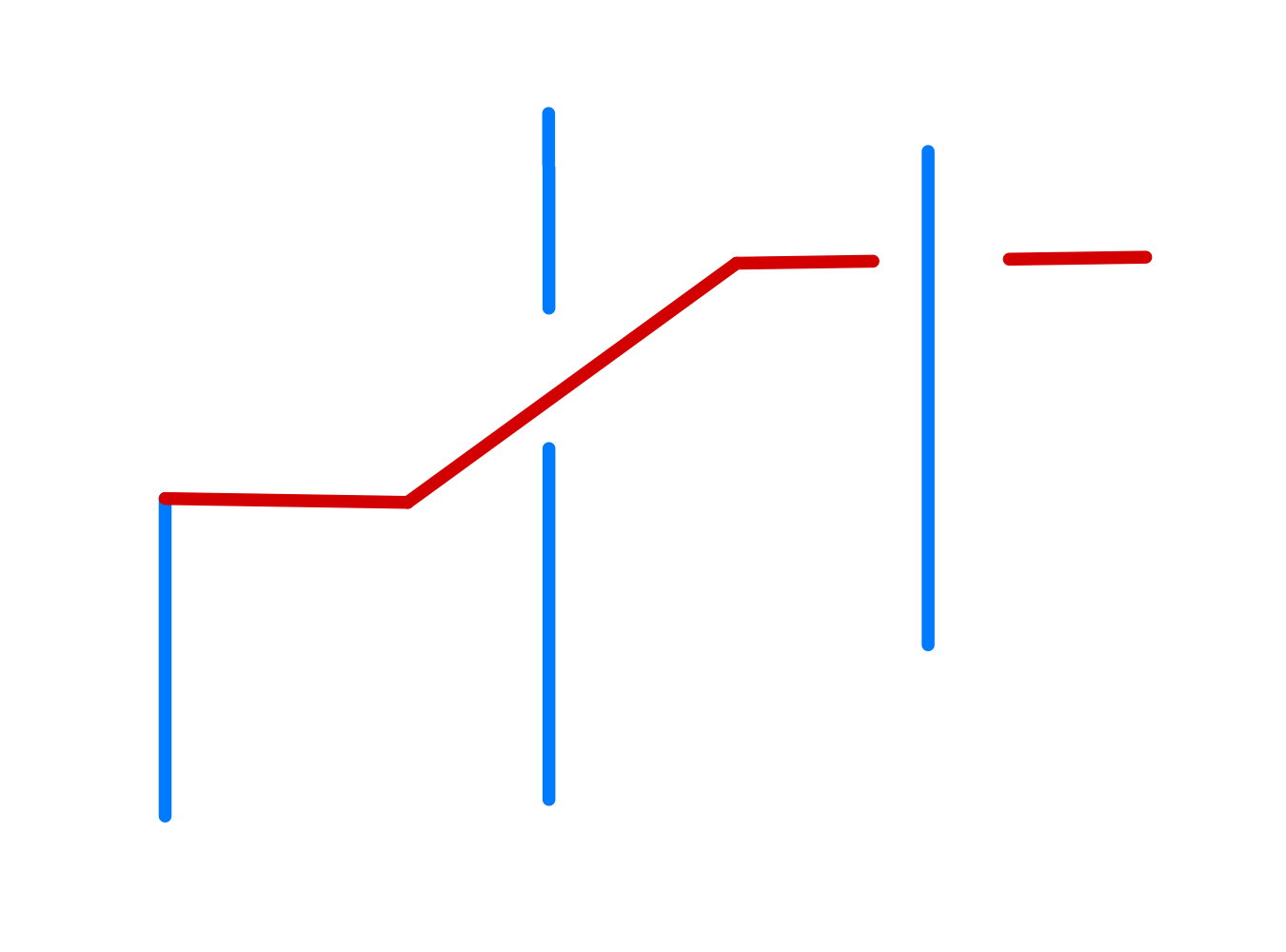}
    \caption{The $xyx$ sequence passes between the two $w$-sticks.}
    \label{xyx1}
\end{figure}
Suppose the remaining $z$-stick is on $w$-level $1$, then the knot projection should look like one of the unknot projections from \fullref{xyx2} below. 
\begin{figure}[ht]
    \centering
    \includegraphics[height=5cm]{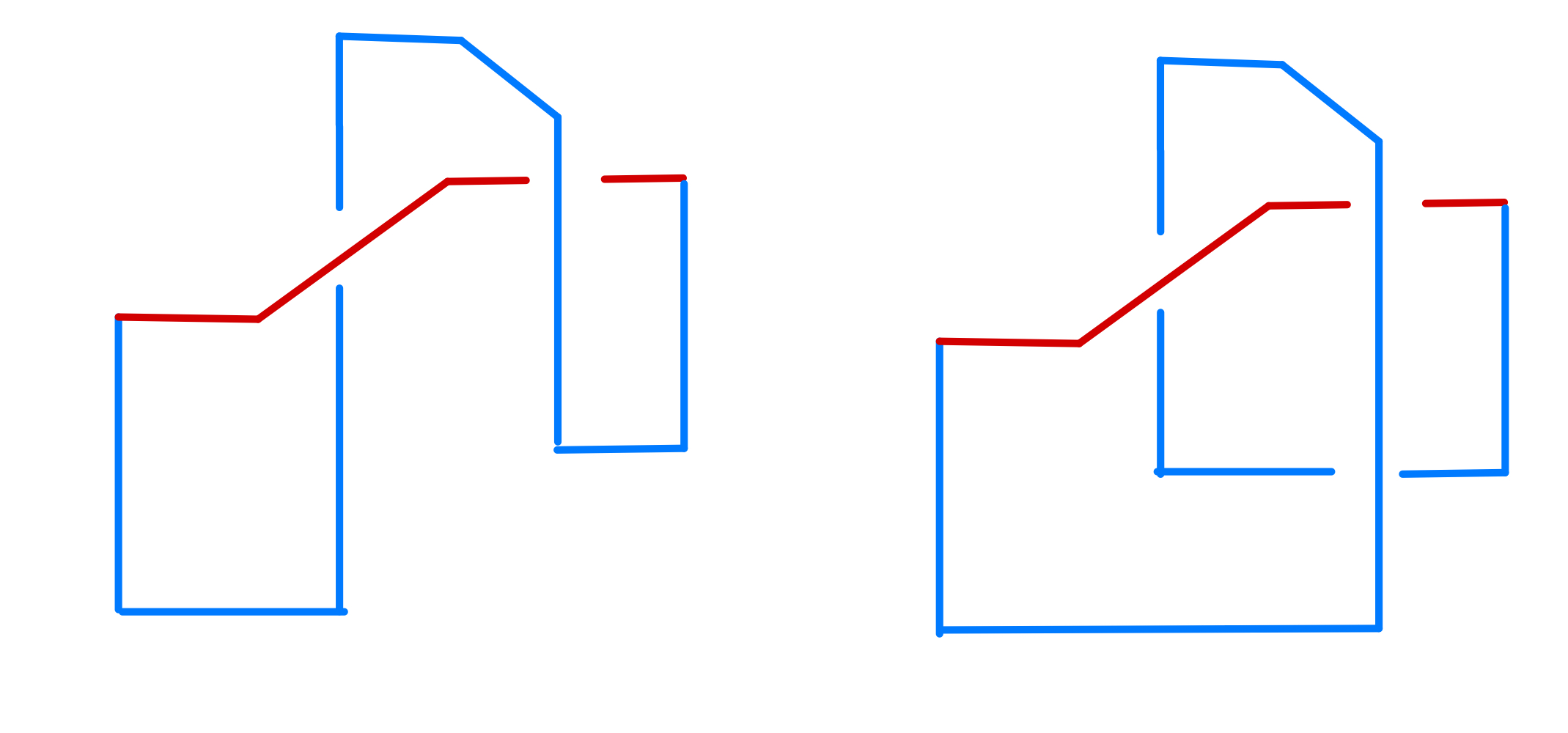}
    \caption{Two projections of the unknot}
    \label{xyx2}
\end{figure}
Otherwise, $w$-level $1$ only has an $x$-stick. Therefore, $w_{13}$ and $w_{14}$ are on the same $x$-level. However, now the two sticks are not in the same $x$-level as the other two $w$-sticks, otherwise we have an intersection of sticks. Therefore, we need at least two $z$-sticks to connect the rest of the knot, a contradiction. 
\end{proof}

A direct consequence of this lemma is the following. 

\begin{corollary}\label{stickbound}
For a non-trivial $11$-stick polygon with four $w$-sticks, $\mathcal P$, there is $|\mathcal P|_x\ge 3$, $|\mathcal P|_y\ge 2$ and $|\mathcal P|_z\ge 1$, up to permutation of stick types. In particular, a non-trivial 11-stick polygon must have either $(a)$ four $x$-sticks, two $y$-sticks, and one $z$-stick, $(b)$ three $x$-sticks, three $y$-sticks and one $z$-stick, or $(c)$ three $x$-sticks, two $y$-sticks and two $z$-sticks without loss of generality. We will refer to these cases as the $(4, 2, 1)$ case, the $(3, 3, 1)$ case, and the $(3, 2, 2)$ case, respectively.
\end{corollary}

\begin{definition} [Square of Replacement]
Consider the view of a sh-lattice knot after applying $T^{-1}$: that is, we send $x$-, $y$-, and $w$-sticks in sh-lattice back to where they were in the cubic lattice, and send the $z$-sticks in sh-lattice to line segments parallel to $\langle-1,1,0\rangle$ in the cubic lattice, while keeping the same labels as in sh-lattice. 
The \textit{square of replacement} is defined to be the square formed with a particular $z$-stick as the diagonal in this view, restricted to the $w$-level on which the $z$-stick sits.
When we say a stick is in the square of replacement, it means the stick intersects with the square at exactly one point. We exclude the endpoints of the $z$-stick from the square of replacement.
\end{definition}

\begin{proposition} \label{z reduction lemma}
Let $\mathcal P$ be any sh-lattice polygon. If there are three $w$-sticks and no $z$-sticks in the square of replacement of a given $z$-stick, and there are at most four sticks on the $w$-level on which the $z$-stick sits, then the $z$-stick can be replaced with $x$- and $y$-sticks, adding a net total of three sticks.
\end{proposition}

\begin{proof}
Consider the figure after applying $T^{-1}$ while keeping the $x$-, $y$-, $z$-, and $w$-sticks labeled as they were. Also, let the top-left corner of the $z$-stick be the origin and let $l$ be the side length of the square of replacement. We will go through the cases for one to three sticks in the square of replacement and show the number of sticks necessary to replace the $z$-stick with $x$- and $y$-sticks does not exceed three. We will also ignore any $x$- or $y$-sticks in the square of replacement, which will be justified at the end of the proof. 

First, assume that we have one $w$-stick in the square. This is either in the upper or lower triangular portion so we can replace the $z$-stick with an $x$- and $y$-stick in the other half of the square, adding one stick. 

Now, assume there are two $w$-sticks in the square (see \fullref{2 w stick case}). If they are both in the upper or lower triangle then we can replace them like above. If not, then there is one on either side. They must differ in the $x$- or $y$-coordinate by at least one. Scale the entire knot by a factor of $2$ so we are guaranteed a gap of at least two between the $w$-sticks. Then, without loss of generality, suppose that they differ in the $y$-coordinate. Let $a$ be the $y$-coordinate of the lowest $w$-stick. Then as depicted in \fullref{2 w stick case}, we can replace the $z$-stick by the following sequence of sticks: $y_{0,a}x_{0,l}y_{a,l}$. 

\begin{figure}[ht]
    \centering
    \includegraphics[width = 0.45\textwidth,trim={0 8cm 0 8cm},clip]{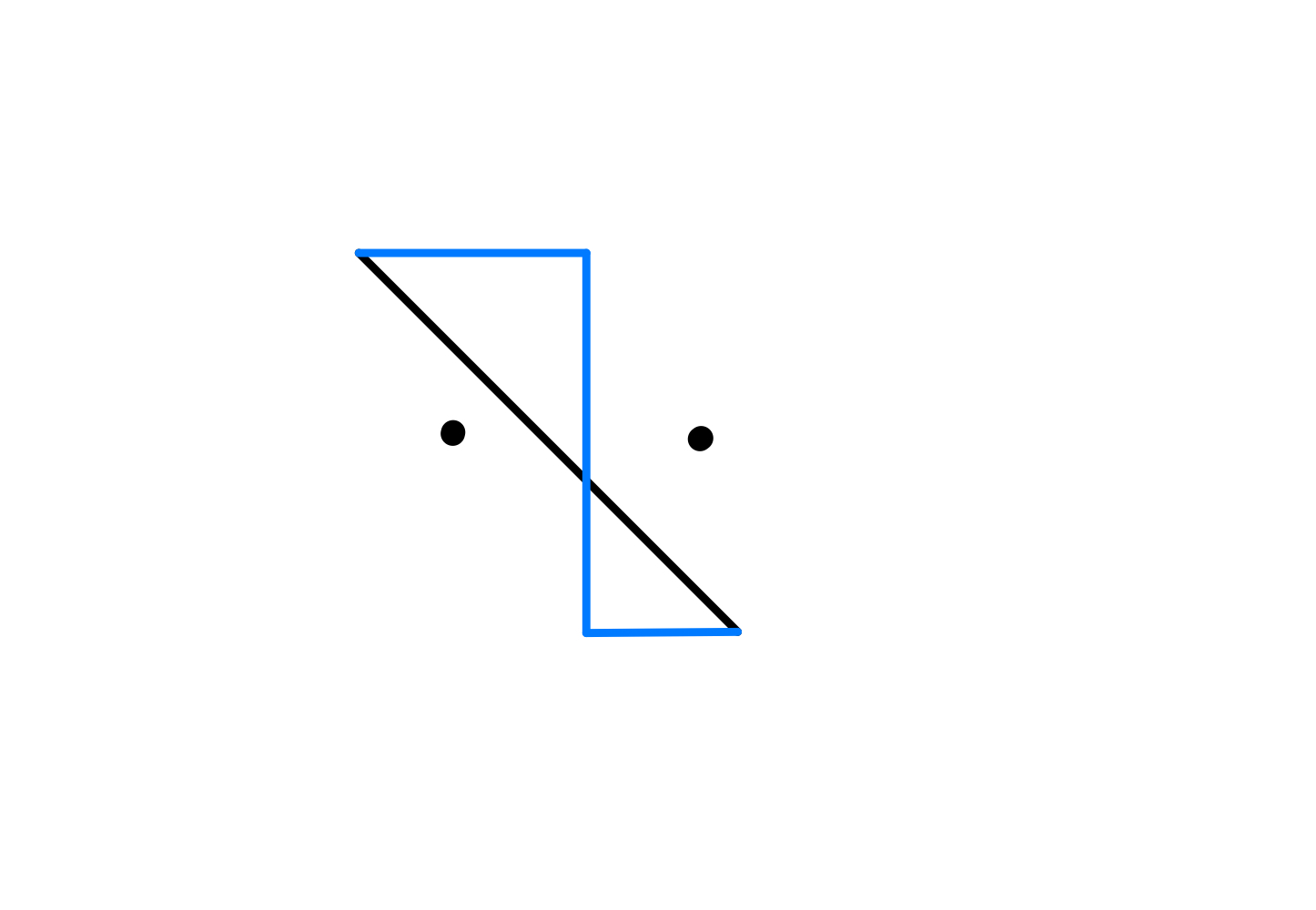}
    \caption{Two $w$-stick case}
    \label{2 w stick case}
\end{figure}

Assume that there are three $w$-sticks in the square. Scale the knot by a factor of $2$. If all $w$-sticks are in the upper or lower triangle then we can replace them like in the one-stick case. Note that between the three of them, there must be at least two unique $x$- or $y$-coordinates. Assume without loss of generality that there are two unique $y$-coordinates then let $s$ be the greatest of the $y$-coordinates. Look at the portions of $z$ above $s$ and below $s$ separately, forming two new squares of replacement for each of these. One of these squares has at most two sticks and the other has at most one. For simplicity, assume that there are the maximum amount of sticks in each square but note that there may be other ways to replace the $z$-stick with fewer sticks if there are fewer than the maximum. 

Without loss of generality, suppose the bottom portion has two $w$-sticks. If each of these portions has both $w$-sticks on one side of the $z$-stick, then we can replace the $z$-stick like in the two $w$-stick case. If not, then the portion with two $w$-sticks has a $w$-stick on either side of the triangle. Once again, they differ in at least one of the $x$- or $y$-coordinates. Thus, we have four subcases for which side of the $z$-stick the $w$-sticks are on and which of the $w$-coordinates we know is different in the lower portion, which we can see in \fullref{3 w stick cases}. However, note that the first and fourth cases are the same up to rotation, and likewise for the second and third. So, we only have two subcases to cover. Also, note that if the $w$-sticks in the lower portion differ in the $x$- and $y$-coordinates, then it satisfies both subcases so the method described in either subcase will work.

For the first subcase, the $w$-stick in the upper square is above the $z$-stick and the lower square has sticks differing in the $x$-coordinate as in the picture. Let $a$ be one less than the $y$-coordinate of the upper $w$-stick and let $b$ be one less than the $x$-coordinate of the rightmost $w$-stick. Then, we can replace $z$ with the following sequence of sticks: $y_{0a}x_{0b}y_{al}x_{bl}$. 

For the second subcase, the $w$-stick in the upper square is below the $z$-stick and the lower square has sticks differing in the $x$-coordinate. Let $a$ be the $x$-coordinate of the rightmost $w$-stick minus $1$. Then, we can replace $z$ with the following sequence of sticks: $x_{0a}y_{0l}x_{bl}$. 

\begin{figure}[ht]
    \centering
    \includegraphics[width = \textwidth,trim={0 8cm 0 8cm},clip]{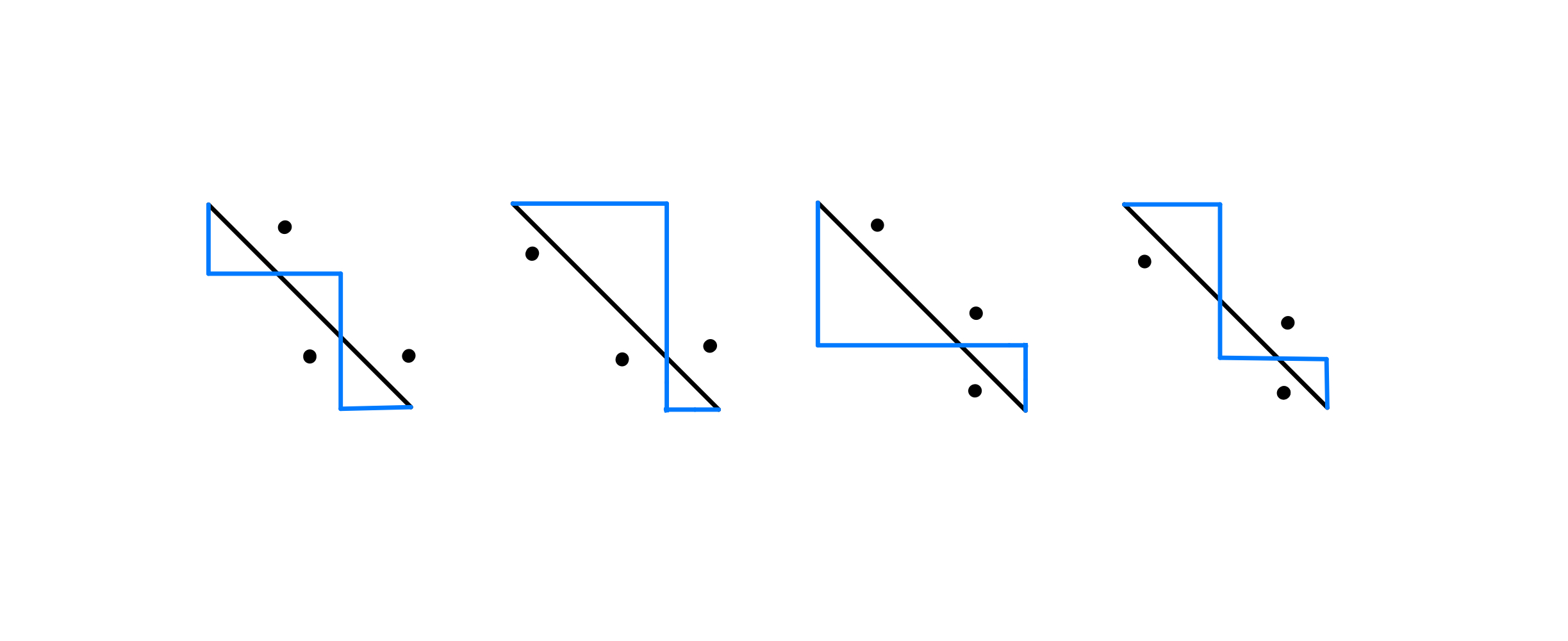}
    \caption{Three $w$-stick cases}
    \label{3 w stick cases}
\end{figure}

We need to address the possibility of $x$- or $y$-sticks in the square on the same $w$-level as the $z$-stick. If we want to replace a $z$-stick with an $x$- and $y$-stick and are obstructed by a stick on the border of the square, we can simply combine them. We will now consider the case where an $x$- or $y$-stick is in the interior of the square. Without loss of generality, suppose that it is an $x$-stick. Note that one endpoint of the $x$-stick must be connected to the $z$-stick. There are two ways that this can occur, namely through a $zxyx$ sequence of sticks or a $zyx$ sequence of sticks. 

In the first case, observe that we know the entirety of the $w$-level and in particular, the endpoint of the $x$-stick intersecting the interior of the square is connected to a $w$-stick. We would like to apply an $R$-move on the $yx$ sequence of sticks. If possible, then we have two outcomes. If the $x$-stick was connected to a $w$-stick outside the square, we no longer have any sticks intersecting the boundary of the square. If not, then the $w$-stick is one of the sticks we consider in the above procedure. In this case, note that having a singular $y$-stick attached to one of the $w$-sticks does not interfere with the procedure, so we can still replace the $z$-stick as desired. Now, suppose that we cannot apply an $R$-move to the $yx$ sequence of sticks. In this case, the $R$-move must be obstructed by a $w$-stick, which must lie within the square of replacement. In this case, we can simply replace that $w$-stick in our procedure with the intersection of the $x$- and $y$-sticks, thus allowing us to replace the $z$-stick as desired.

In the second case, if the $x$-stick is connected to a $w$-stick, then the $x$-stick plays no role and we can replace the $z$-stick like normal. Otherwise, we have a $zyxy$ sequence of sticks. The same logic applies as when we had the $zxyx$ sequence of sticks so once again, we can replace the $z$-stick as desired.

As a result of this procedure, there will be at most a net total of three sticks added to our knot (i.e. one stick removed and four added).
\end{proof}

\begin{remark}
    For any $11$-stick polygon, there are at most four sticks in any given $w$-level so \fullref{z reduction lemma} is always applicable.
\end{remark}

\begin{lemma} \label{4wstick z reduction}
Let $\mathcal{P}$ be a $11$-stick polygon with four $w$-sticks. Then, if a $z$-stick has four $w$-sticks in its square of replacement, the $z$-stick can be replaced with $x$- and $y$-sticks with an addition of at most three sticks.
\end{lemma}

\begin{proof}
If none of the $w$-sticks are on the border of the square, then from each end of our $z$-stick, we need two sticks to reach another $w$-stick. Then, each other $w$-stick requires at least one stick, meaning that we use seven sticks in total, which takes us to twelve total sticks in total. If there is one $w$-stick on the border of the square, we can reach it from the end of the $z$-stick with one stick while the $w$-stick connected to the other endpoint of the $z$-stick will require two. If we can connect all the rest with one stick, we get eleven sticks in total. No other configurations are possible so we have this configuration of sticks. Note that we can use an $R$-move to move the stick on the border to the endpoint of the $z$-stick. Now by \fullref{z reduction lemma} we can reduce the $z$-stick with the addition of at most three sticks. Lastly, suppose there are two $w$-sticks on the border. If neither of them is connected to the endpoints of the $z$-stick with one stick then we run into the same issue as if there were two. If one of them is, then we have the same situation as the case with one on the border. If both are connected to the $z$-stick endpoints with one stick, note that we can connect all our required points with one stick each to get ten sticks. Thus, we have one stick unaccounted for. However, this stick can only add one more endpoint to our understanding of the knot so it can obstruct $R$-moves at most at one of the endpoints of the $z$-stick; if we had a stick obstructing $R$-moves at both ends, we need more than one stick beyond what is already required to connect everything else. So, on at least one of the endpoints, we can perform an $R$-move to move a $w$-stick onto the endpoint of the $z$-stick and we have the three $w$-stick case.
\end{proof}

\begin{remark}
Note that by the result in \cite{huang2017lattice}, we know that any knot in the cubic lattice with less than sixteen sticks is a trefoil knot or a figure-eight knot. Therefore, we can add up to four sticks in our replacement of $z$-sticks starting from a non-trivial $11$-stick polygon, and the resultant knot is still a $3_1$ or $4_1$ knot.
\end{remark}

\begin{corollary}\label{1least}
The only non-trivial $11$-stick polygons $\mathcal P$ with four $w$-sticks satisfying 
$$
\min\{|\mathcal P|_x,|\mathcal P|_y,|\mathcal P|_z\} = 1$$
are $3_1$ and $4_1$. 
In other words, any knots in the $(4, 2, 1)$ or the $(3, 3, 1)$ cases are $3_1$ or $4_1$.
\end{corollary}

\begin{proof}
Note that we have only four $w$-sticks in an $11$-stick polygon so we don't have to account for cases with more than four $w$-sticks in the square.

Without loss of generality, we can say that a polygon in question has one $z$-stick. Then, by \fullref{4wstick z reduction}, this $z$-stick can be reduced into $x$- and $y$-sticks with the addition of at most three $x$- or $y$-sticks. Hence, the knot has at most fourteen sticks and is in the cubic lattice (by applying $T^{-1}$), so it is either the trefoil knot or the figure-eight knot.
\end{proof}

Now that we know any $11$-stick polygon with four $w$-sticks and one $z$-stick is $3_1$ or $4_1$, we seek to prove the same statement for $11$-stick polygons with four $w$-sticks and two $z$-sticks. First, we identify some conditions these pairs of $z$-sticks must meet. 

\begin{lemma}\label{diff w levels}
For any irreducible knot of the case $(3,2,2)$ that is not of type $3_1$ or $4_1$, we can rotate the knot so the $z$-sticks are on different $w$-levels.
\end{lemma}
\begin{proof}
First, note that if the $y$-sticks are on different levels, we can rotate the knot so that the $y$-sticks are now $z$-sticks. If not, then note first that the $z$-sticks cannot both be on the boundary levels because we would be able to combine them into one stick, or there would have to be a third stick in between them, both of which would be a contradiction. So without loss of generality, suppose they are on level $2$. First, suppose that both $y$-sticks are on $w$-level $3$. Then, we need an $x$-stick on each boundary, leaving one stick for the inner levels. However, both levels need an $x$-stick; otherwise, there would only be one type of stick on a level and we could combine them to reduce the stick number. Now, suppose that the $z$- and $y$-sticks are on level $2$. Then, we must have one $x$-stick on each other level. Rotate the knot so that the $x$-sticks are now $z$-sticks and apply $T^{-1}$. On both boundary levels, we can replace the $z$-stick with $x$- and $y$-sticks, increasing the stick number by $1$ each. On $w$-level $3$, we have only one stick so the $w$-sticks must be on the endpoints of the $x$-sticks. Thus, there are only two $w$-sticks unaccounted for and we can replace them like in the two $w$-stick case from \fullref{z reduction lemma}, adding at most two sticks. So, we have a knot in the cubic lattice with at most fifteen sticks, indicating that it is $3_1$ or $4_1$. 
\end{proof}

\begin{lemma} \label{not 2}
If there are two $z$-sticks in an irreducible $11$-stick knot with four $w$-sticks, then they cannot both be in the three $w$-stick case from \fullref{z reduction lemma}. 
\end{lemma}

\begin{proof}
Suppose we had two $z$-sticks in the three $w$-stick case. Then, there are two possibilities: either we have at least three $w$-sticks unconnected to $z$-sticks, or two $w$-sticks are unconnected to $z$-sticks. 

In the first case, we need at least seven sticks to connect at least seven different points, two of which are provided by the $z$-sticks. So, we can only have one stick connecting every pair of points. Thus, all sticks are accounted for and we can move a $w$-stick to the endpoint of a $z$-stick and at least one of the sticks is now in the two $w$-stick case. 

In the second case, note that each of the $z$-sticks must have a $w$-stick on each endpoint (because this stick along with the two others forms the three $w$-stick case for the other $z$-stick). Also, note that one $z$-stick has to be on $w$-level $2$ and the other on $w$-level $3$. Looking at the $z$-stick on $w$-level $2$, there are two possibilities for the $w$-stick on its endpoint: either the $w$-stick is the $w_{23}$ stick or the $w_{24}$ stick. If it is the $w_{24}$ the other $w$-stick sits right on top of it. However, this means that the $w$-sticks on the endpoints can't serve as the third obstruction so we only have the two $w$-stick case. If it is the $w_{24}$ stick then the $w_{23}$ stick must be one of the obstructing $w$-sticks. So, we need to connect the other endpoints of the $w$-stick to the $w_{23}$ stick. Note that these endpoints are not within the overlapping region of both $z$-sticks' squares. So, it will take at least two sticks to connect to $w_{23}$ from each $z$-stick. Also, it will take at least one stick each to connect the $w_{14}$ stick to the $w_{13}$ and $w_{24}$ sticks. However, this takes us up to twelve sticks in the knot which is a contradiction. 
\end{proof}

At this point, we define the case on which we will focus for the next few lemmas. 

\begin{definition}[$p$-case] \label{p-case}
For simplicity, we will call the first case detailed in \fullref{3 w stick cases}, which is the only case that requires a net gain of three sticks, the problematic case, or \textit{$p$-case} for short.
\end{definition}

\begin{remark}
In the next few lemmas, when we say a stick is ``attached to an $x$-, $y$-, or $z$-stick'', we are only looking at the knot projection onto the $xy$-plane and allow the possibility that there is a $w$-stick between them. 
\end{remark}

\begin{lemma}\label{a}
Suppose we have an irreducible $11$-stick knot in the sh-lattice with four $w$-sticks. If we have
 a $z$-stick in the $p$-case and it is attached to another $z$-stick, then the knot must be $3_1$ or $4_1$.
\end{lemma}

\begin{proof}
If the $z$-stick not in the $p$-case is in the one-stick case, then we know from \fullref{z reduction lemma} that we can replace it by adding one stick and we can replace the stick in the $p$-case with a net gain of three sticks. This gives us a knot with fifteen sticks so it must be a trefoil knot or a figure-eight knot. We will now show that we cannot have a stick in the $p$-case and another in the two $w$-stick case. 

First, assume that the $z$-sticks are connected and point in different directions from the $w$-stick connecting them so they do not lie on top of each other. Then, we only have three more $w$-sticks so we cannot have one with two obstructions and another with three. 
Now, assume that the $z$-sticks are connected and point in the same direction from the $w$-stick connecting them so they lie on top of each other. Then, we can perform an $R$-move between the $w$-stick and the overlapping portion of the $z$-stick then combine the overlapping $z$-sticks so that the $z$-sticks are in the earlier case. 
\end{proof}

\begin{lemma} \label{b}
 Suppose we have an irreducible $11$-stick knot in the sh-lattice with four $w$-sticks. If we have a $z$-stick in the $p$-case and the bottom right corner is attached to a $y$-stick with its other endpoint higher than the shared $y$-coordinate, then the knot is either $3_1$ or $4_1$. 
\end{lemma}

\begin{proof}
Suppose we have a $z$-stick in the situation described in the lemma. Then, note that we need the following sticks in our knot:
\begin{itemize}
    \item a $z$-stick, 
    \item a $y$-stick attached to the bottom end of the $z$-stick, 
    \item at least two sticks to attach the other endpoint of this $y$-stick to one of the $w$-sticks on the lower level, 
    \item at least one stick to attach the two $w$-sticks on the lower level, 
    \item at least two sticks to attach the lower level $w$-sticks back to the upper endpoint of the $z$-stick. 
\end{itemize}
These restraints give us the cases in \fullref{Cases}. 

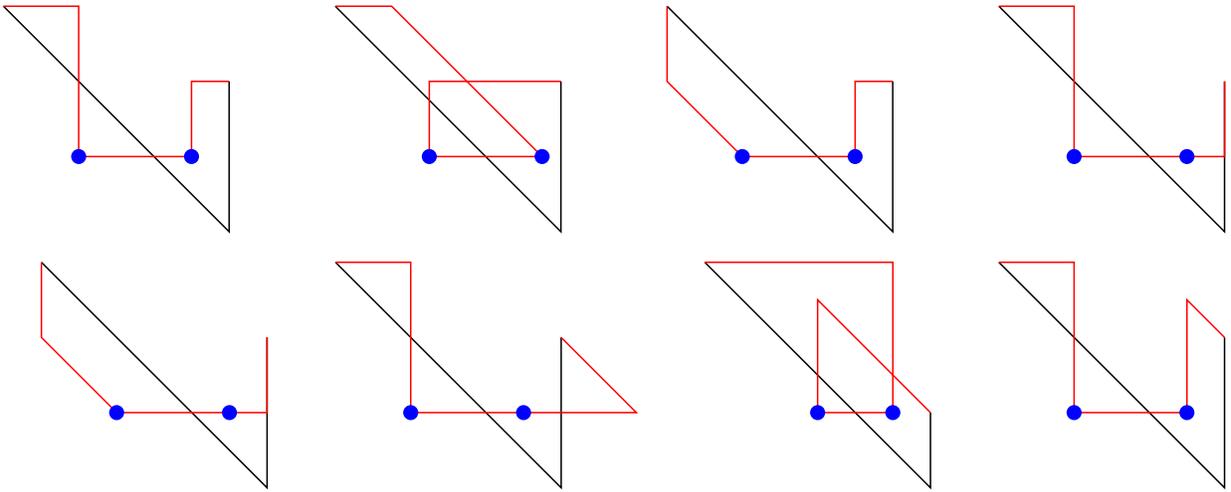
\begin{figure}[ht]
    \centering
    \subfigure{
    \begin{tikzpicture}
    \draw[line width=0.2mm] (0,0) -- ++(3,-3) -- ++(0,2);
    \draw[line width = 0.2mm, red] (3,-1)-- ++(-0.5,0) -- ++(0,-1)--++(-1.5,0) -- ++(0,2) --++(-1,0);
    \fill[blue] (1,-2) circle(.1);
    \fill[blue] (2.5,-2) circle(.1);
    \end{tikzpicture}
    }
    \hfill
    \subfigure{
    \begin{tikzpicture}
    \draw[line width=0.2mm] (0,0) -- ++(3,-3) -- ++(0,2);
    \draw[line width = 0.2mm, red] (3,-1)-- ++(-1.75,0) -- ++(0,-1) -- ++(1.5,0) -- ++(-2,2) --++(-0.75,0);
    \fill[blue] (1.25,-2) circle(.1);
    \fill[blue] (2.75,-2) circle(.1);
    \end{tikzpicture}
    }
    \hfill 
    \subfigure{
    \begin{tikzpicture}
    \draw[line width=0.2mm] (0,0) -- ++(3,-3) -- ++(0,2);
    \draw[line width = 0.2mm, red] (3,-1)-- ++(-0.5,0) -- ++(0,-1)--++(-1.5,0) -- ++(-1,1) --++(0,1);
    \fill[blue] (1,-2) circle(.1);
    \fill[blue] (2.5,-2) circle(.1);
    \end{tikzpicture}
    }
    \hfill 
    \subfigure{
    \begin{tikzpicture}
    \draw[line width=0.2mm] (0,0) -- ++(3,-3) -- ++(0,2);
    \draw[line width = 0.2mm, red] (3,-1)-- ++(0,-1) -- ++(-2,0) -- ++(0,2) --++(-1,0);
    \fill[blue] (1,-2) circle(.1);
    \fill[blue] (2.5,-2) circle(.1);
    \end{tikzpicture}
    }\\
    \hfill 
    \subfigure{
    \begin{tikzpicture}
    \draw[line width=0.2mm] (0,0) -- ++(3,-3) -- ++(0,2);
    \draw[line width = 0.2mm, red] (3,-1)-- ++(0,-1) -- ++(-2,0)--++(-1,1) -- ++(0,1);
    \fill[blue] (1,-2) circle(.1);
    \fill[blue] (2.5,-2) circle(.1);
    \end{tikzpicture}
    }
    \hfill 
    \subfigure{
    \begin{tikzpicture}
    \draw[line width=0.2mm] (0,0) -- ++(3,-3) -- ++(0,2);
    \draw[line width = 0.2mm, red] (3,-1)-- ++(1,-1) -- ++(-3,0) -- ++(0,2) --++(-1,0);
    \fill[blue] (1,-2) circle(.1);
    \fill[blue] (2.5,-2) circle(.1);
    \end{tikzpicture}
    }
    \hfill
    \subfigure{
    \begin{tikzpicture}
    \draw[line width=0.2mm] (0,0) -- ++(3,-3) -- ++(0,1);
    \draw[line width = 0.2mm, red] (3,-2) --++(-1.5,1.5) -- ++(0,-1.5) -- ++(1,0) -- ++(0,2) -- ++(-2.5,0);
    \fill[blue] (1.5,-2) circle(.1);
    \fill[blue] (2.5,-2) circle(.1);
    \end{tikzpicture}
    }
    \hfill
    \subfigure{
    \begin{tikzpicture}
    \draw[line width=0.2mm] (0,0) -- ++(3,-3) -- ++(0,2);
    \draw[line width = 0.2mm, red] (3,-1)-- ++(-0.5,0.5) -- ++(0,-1.5)--++(-1.5,0) -- ++(0,2) --++(-1,0);
    \fill[blue] (1,-2) circle(.1);
    \fill[blue] (2.5,-2) circle(.1);
    \end{tikzpicture}
    }
    \caption{The possible cases}
    \label{Cases}
\end{figure}

Note that in all of these, we have less than five crossings so it must be either the trefoil knot, the figure-eight knot, or the unknot (which would be a contradiction).
\end{proof}

\begin{lemma}\label{pre}
 Suppose we have an irreducible $11$-stick knot in the sh-lattice with four $w$-sticks. Also, suppose we have a $z$-stick in the $p$-case and the bottom right corner is attached to a $x$-stick or a $y$-stick with its other endpoint lower than or equal to the shared $y$-coordinate, then the knot must be either $3_1$ or $4_1$. 
\end{lemma}
\begin{proof}
In the first case, we can follow the procedure in \fullref{z reduction lemma} to replace the $z$-stick with a $yxyx$ sequence of sticks. Because the last stick in this sequence is a $x$-stick, we can combine it with the preexisting $x$-stick so we only have a net gain of two sticks, as we can see in the first case of \fullref{Cases2}.
Thus, we get a knot in the cubic lattice with at most fifteen sticks so it must either be the trefoil knot or the figure-eight knot.

In the second case, where the $y$-stick has an endpoint lower than the $y$-level of the shared $y$-coordinate, we can replace the $z$- and $y$-sticks with a $yxyx$ sequence, as we can see in case $2$ of \fullref{Cases2}, adding four sticks but removing two for a net gain of two sticks. 

Now suppose the $y$-stick connected to the bottom right corner of the $z$-stick has an endpoint at the same $y$-level as the shared $y$-coordinate.

If the other endpoint of the $y$-stick is connected to one of our two $w$-sticks in the lower portion of the $p$-case, then we can replace both the $y$- and $z$-stick with a $yxy$ sequence of sticks as we can see in case $3$ of \fullref{Cases2}. Otherwise, we have three cases depending on what stick the $y$-stick is connected to:

If its other endpoint is a $y$-stick, we can either combine them or do an $R$-move to change the endpoint so it is no longer equal to the two lower $w$-sticks and we can refer to an earlier case. 

If this endpoint is a $z$-stick, we can scale the knot by a factor of $2$. Then, move the arc from this $z$-stick until the next $y$-stick upwards by one. Now, we are in the \fullref{b} case. 

If the endpoint is connected to an $x$-stick, then this endpoint must be connected to the rightmost $w$-stick. Then, we can once again replace the $z$-stick, $y$-stick, and $x$-stick with a $yxy$ sequence of sticks as in case $4$ of \fullref{Cases2}.

Note that if there is a $w$-stick in the way of these replacements, then we need to do an $R$-move to place the sticks on the same $w$-level. There cannot be a stick obstructing this $R$-move as this would require at least two more sticks to have a stick intersecting one such $R$-move which would exceed the total of eleven sticks.

Thus, we have a net gain of at most two sticks. The other $z$-stick is in the one- or two $w$-stick case by \fullref{not 2} so we can replace it with at most two sticks as well. Thus, we have a knot in the cubic lattice with at most fifteen sticks so it is a trefoil knot or a figure-eight knot. 
\end{proof}

\begin{figure}[ht]
    \centering
    \includegraphics[width=7cm]{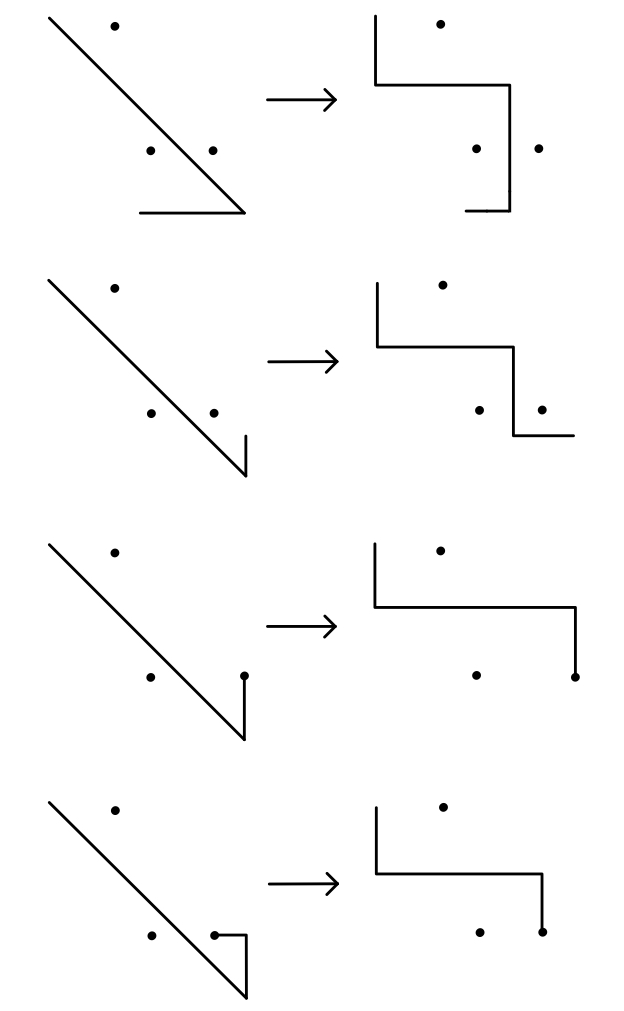}
    \caption{The cases from \fullref{pre} in which a $z$-stick is replaced}
    \label{Cases2}
\end{figure}

\begin{lemma}\label{2least}
The only non-trivial $11$-stick polygons $\mathcal P$ with four $w$-sticks satisfying 
$$\min\{|\mathcal P|_x,|\mathcal P|_y,|\mathcal P|_z\} = 2$$
have knot type $3_1$ or $4_1$. 
\end{lemma}

\begin{proof}
When we have a knot in the $(3, 2, 2)$ case, by \fullref{diff w levels}, the $z$-sticks are on different levels. So, we can apply the procedures from \fullref{z reduction lemma} to create a presentation of the knot in the cubic lattice with less than sixteen sticks unless one or both of our $z$-sticks is in the $p$-case defined in \fullref{p-case}. However, from \fullref{a}, \fullref{b}, and \fullref{pre}, we showed that a non-trivial knot with a stick in the $p$-case is either $3_1$ or $4_1$. 
\end{proof}

Finally, we will consider the case where a $11$-stick polygon has five $w$-sticks. 

\begin{lemma}\label{5 w}
The only non-trivial $11$-stick polygons $\mathcal P$ with five $w$-sticks are $3_1$ and $4_1$. 
\end{lemma}
\begin{proof}
    Let $\mathcal{P}$ be a non-trivial $11$-stick polygon with five $w$-sticks. Then, we have three cases, namely the cases $(4,1,1)$, $(3,2,1)$, or $(2,2,2)$ with notation as in the four $w$-stick case. Note that in all cases, we have at most two sticks on every $w$-level. We will first establish that we can replace any $z$-stick with the net addition of at most three $x$- or $y$-sticks in this scenario. 
    
    If a $z$-stick is the only stick on its level, then there can be at most three $w$- sticks in its square of replacement. So, by \fullref{z reduction lemma}, we can replace it with a net addition of three sticks. Now, we will consider the case where a $z$-stick is connected to another stick, which we will assume is an $x$-stick without loss of generality. Observe that we can replace the intersection of the $z$-stick and the $x$-stick with a small $y$-stick, as illustrated in \fullref{add-y}, similarly to what was done in \fullref{reduceedge}. In particular, let the length of the $y$-stick be small enough that the rectangle with the sides given by the $x$-stick and this $y$-stick contains no $w$-sticks. Then, looking at the remaining $z$-stick, we only have three $w$-sticks in the square of replacement, so we can apply the procedure from \fullref{z reduction lemma}. If the $y$-stick we added is connected to a another $y$-stick, we can combine them, meaning that we successfully replaced the $z$-stick with a net addition of at most three sticks. If not, then it is connected to another $x$-stick. In this case, note that we chose the $y$-stick we inserted to be sufficiently small that we can perform an $R$-move which reduces the number of sticks by one. So, we can replaced the $z$-stick with a net addition of at most three sticks. 

    \begin{figure} [ht]
    \centering
    \begin{tikzpicture}
        \draw[line width = 0.4mm] (7.5,-0.5) --++(4.5,-4.5);
        \draw[line width = 0.4mm] (7.5, 0) --++(1.5, 0);
        \draw[line width = 0.4mm] (7.5, 0) --++(0, -0.5);
        \fill[black!100] (2,0) circle(.08);
        \fill[black!100] (3.5,-1) circle(.08);
        \fill[black!100] (4.5,-3) circle(.08);
        \fill[black!100] (1.5,-3) circle(.08);
        \draw[line width = 0.5mm, ->] (5,-2.5) -- ++ (1,0);
        \draw[line width = 0.4mm] (0,0) --++(5,-5);
        \draw[line width = 0.4mm] (0, 0) --++(2, 0);
        \fill[black!100] (9,0) circle(.08);
        \fill[black!100] (10.5,-1) circle(.08);
        \fill[black!100] (11.5,-3) circle(.08);
        \fill[black!100] (8.5,-3) circle(.08);
    \end{tikzpicture}
    \caption{Addition of $y$-stick}
    \label{add-y}
    \end{figure}
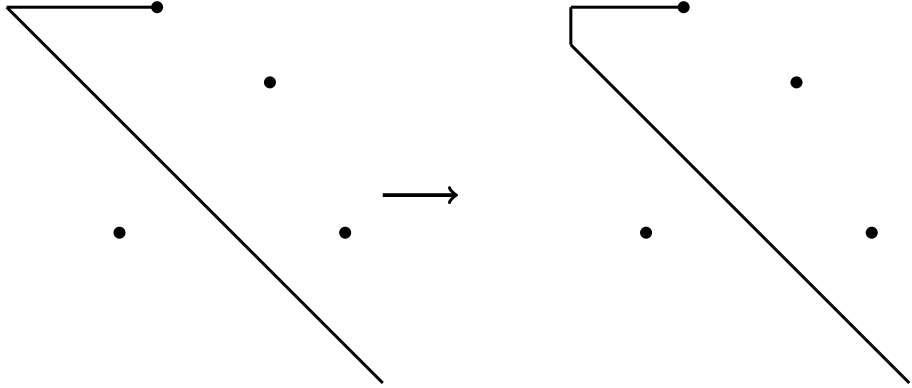
    
    Now, observe that in the cases $(4,1,1)$ and $(3,2,1)$, we can replace the $z$-stick with at most three $x$- or $y$-sticks, allowing us to create a presentation in the cubic lattice with at most fourteen sticks. In the $(2,2,2)$ case, we can assume without loss of generality that we have a $z$-stick on the $w$-level $1$. We can replace this stick with two sticks, resulting in a net addition of one stick. We can replace the other $z$-stick with a net addition of three sticks. So, we can create a polygon in the cubic lattice with at most fifteen sticks. So, in all cases, the knot is either $3_1$ or $4_1$. 
\end{proof}

We have now discussed all the possible cases for an irreducible $11$-stick polygon in the sh-lattice. Collecting the properties we proved from \fullref{stickbound}, \fullref{1least}, \fullref{2least}, and \fullref{5 w} above, we may state the desired theorem.

\begin{theorem}\label{finaltheorem}
In the sh-lattice, the only non-trivial $11$-stick knots are $3_1$ and $4_1$. 
\end{theorem}

\section{Future Work}

We know from \cite{huang2017lattice} that every knot type $[K]$ with crossing number $c[K]\ge 5$ has stick number $s_L[K]\ge 16$ in the cubic lattice. Therefore, it is possible to restrict the bound in \fullref{finaltheorem} even further. One can easily conclude from \fullref{51knot} and \fullref{52knot} that both $5_1$ and $5_2$ have stick number at most $14$. 
\begin{figure}[ht]
    \centering
    \subfigure[$14$-stick $5_1$ projection]{\includegraphics[width = .45\textwidth]{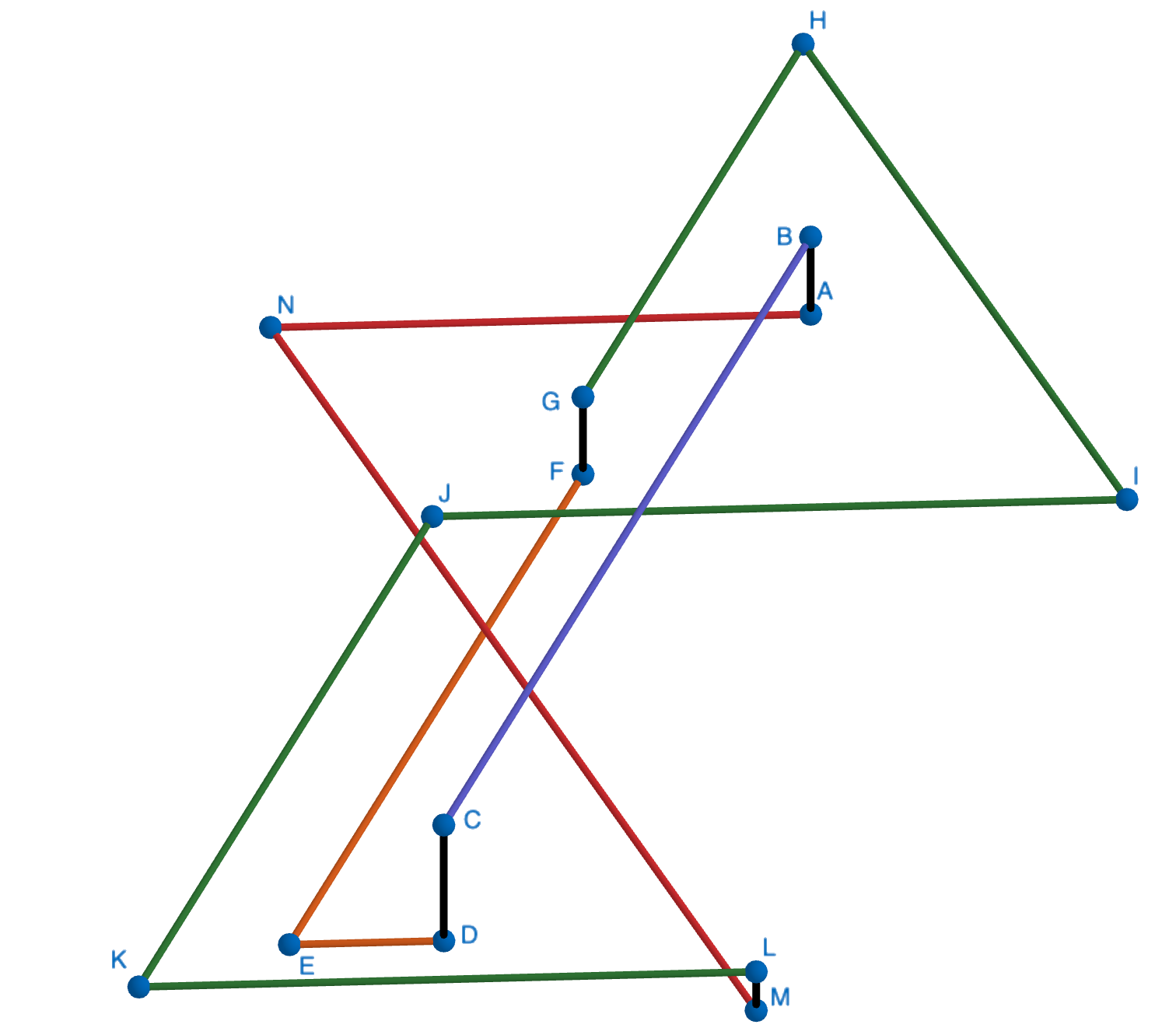}\label{51knot}}
    \subfigure[$14$-stick $5_2$ projection]{\includegraphics[width = .45\textwidth]{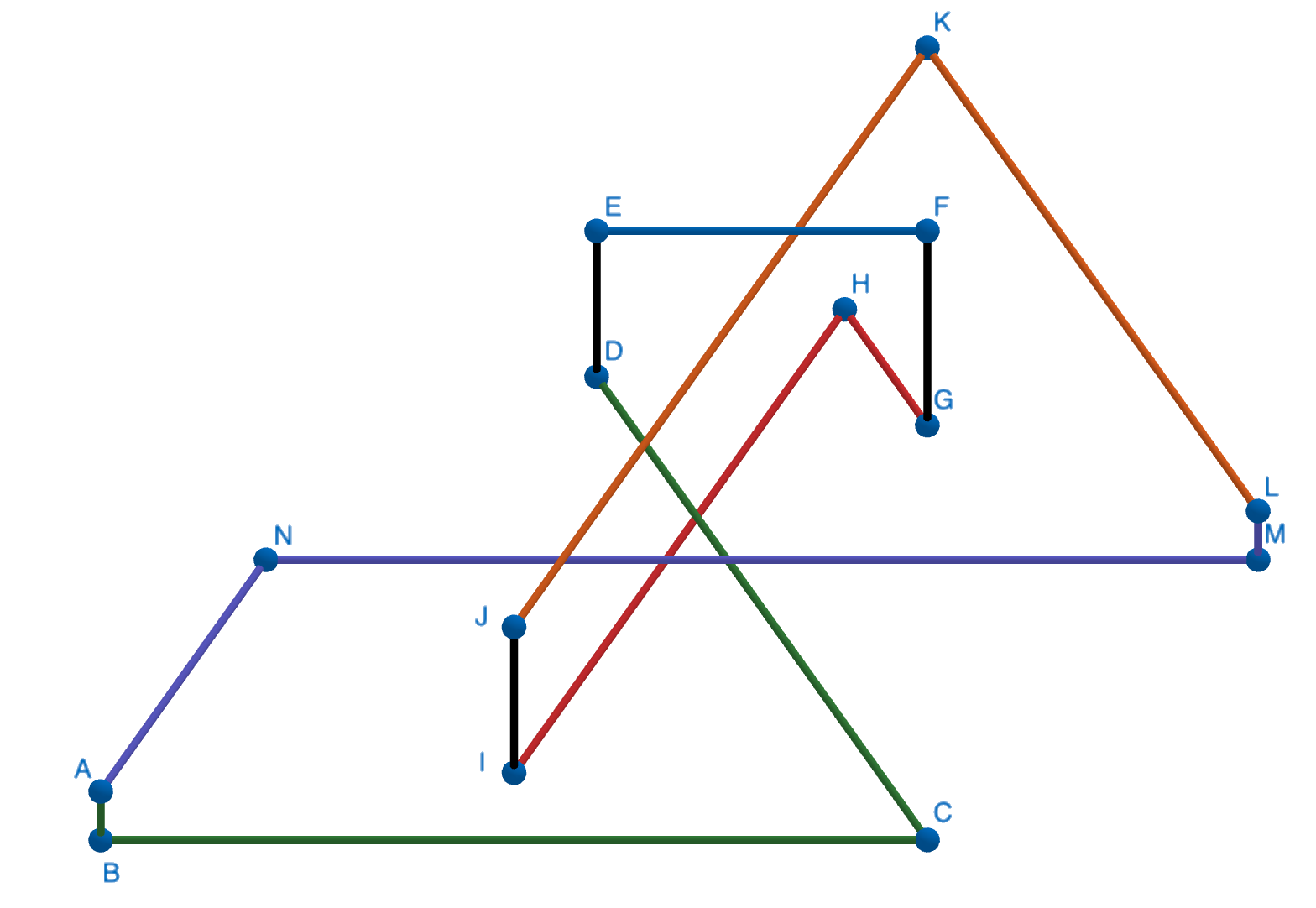}\label{52knot}}
    \caption{The stick numbers of $5_1$ and $5_2$ are at most $14$}
    \label{14stick}
\end{figure}
Hence, we have the following conjecture. 

\begin{conjecture}
The only non-trivial knot types $[K]$ with stick number $s_{sh}[K] \le 13$ in the cubic lattice are $3_1$ and $4_1$. In particular, $s_{sh}(5_1) = s_{sh}(5_2) = 14$. 
\end{conjecture}

Similarly, we have the following conjecture. This, however, likely requires improved lower bounds for $6$-crossing knots in the cubic lattice before it can be seriously approached.

\begin{conjecture}
For a knot type $[K]$ such that $s_{sh}[K] \leq 14$, then $c[K] \leq 5$, where $c[K]$ is the crossing number of a knot type $[K]$.
\end{conjecture}

Another potential direction to work on in the future is to further investigate the relationships between the number of $w$-sticks and the number of sticks in general. This may help us classify sh-lattice knots with a specific number of sticks, as it is heavily involved in the proof of \fullref{specific knots} for instance. 

\begin{problem}
For a properly leveled polygon $\mathcal P$ of type $[K]$, construct upper and lower bounds on the number of $w$-sticks, both in terms of stick number $s_{sh}[K]$ and in terms of crossing number $c[K]$.
\end{problem}

As a call back to \fullref{4.10}, we also have the following direction for further research.

\begin{problem}
Improve the lower bound on $e_{sh}$ in terms of $e_L$, and find a lower bound on $s_{sh}$ in terms of $s_L$. 
\end{problem}

\section*{Acknowledgements}

Many thanks to the organizers of \href{https://geometrynyc.wixsite.com/polymathreu}{Polymath Jr.} who made it possible for this group to meet and do this work.

\printbibliography

\end{document}